\documentclass[a4paper, 11pt]{article}
\usepackage[T1]{fontenc}
\usepackage[utf8]{inputenc}
\usepackage[all]{xy}
\usepackage{amsfonts}
\usepackage{graphicx}

\usepackage{enumerate}		
\usepackage{hyperref}		
\usepackage{relsize}
\usepackage{amssymb} 		
\usepackage{mathrsfs} 		
\usepackage{bbm} 		
\usepackage{extarrows} 		
\usepackage{stmaryrd} 		
\usepackage{mathtools}		
\usepackage{graphicx}
\usepackage{amsmath}	
\usepackage{amsthm} 		
\usepackage[nottoc]{tocbibind} 

\usepackage{xurl}     
\usepackage{hyperref} 

\theoremstyle{plain}
\newtheorem{thm}{Theorem}
\numberwithin{thm}{section}
\newtheorem{prop}[thm]{Proposition}

\newtheorem{cor}[thm]{Corollary}

\theoremstyle{definition}
\newtheorem{defn}{Definition}
\numberwithin{defn}{section}

\newtheorem{ex}{Example}
\numberwithin{ex}{section}

\numberwithin{nota}{section}

\newtheorem{remark}{Remark}
\newtheorem{remarks}[remark]{Remarks}
\numberwithin{remark}{section}

\usepackage{amsthm}

\usepackage{adjustbox}
\usepackage{tikz}
\usepackage{tikz-cd} 
\usetikzlibrary{positioning,intersections,decorations.pathreplacing} 
\usetikzlibrary{decorations.pathmorphing} 
\usetikzlibrary{arrows}
\usepackage{chngpage}

\usepackage{comment}

\DeclareFontFamily{U}{min}{}
\DeclareFontShape{U}{min}{m}{n}{<-> dmjhira}{}

\usepackage{xspace} 

\newcommand{\ac}{`}

\renewcommand{\epsilon}{\varepsilon}
\renewcommand{\phi}{\varphi}

\renewcommand{\lim}{\operatorname{lim}}
\newcommand{\colim}{\operatorname{colim}}

\newcommand{\comma}[2]			
{\mbox{$(#1\!/\!#2)$}}

\newcommand{\Cat}{\mathbf{Cat}}

\newcommand{\Sh}{\mathbf{Sh}}
\newcommand{\St}{\mathbf{St}}

\newcommand{\cc}{\mathcal{C}}
\newcommand{\cd}{\mathcal{D}}
\newcommand{\ce}{\mathcal{E}}
\newcommand{\cf}{\mathcal{F}}
\newcommand{\cg}{\mathcal{G}}

\newcommand{\mbc}{\mathbb{C}}
\newcommand{\mbd}{\mathbb{D}}

\makeatletter
\newbox\xrat@below
\newbox\xrat@above
\newcommand{\xrightarrowtail}[2][]{%
	\setbox\xrat@below=\hbox{\ensuremath{\scriptstyle #1}}%
	\setbox\xrat@above=\hbox{\ensuremath{\scriptstyle #2}}%
	\pgfmathsetlengthmacro{\xrat@len}{max(\wd\xrat@below,\wd\xrat@above)+.6em}%
	\mathrel{\tikz [>->,baseline=-.75ex]
		\draw (0,0) -- node[below=-2pt] {\box\xrat@below}
		node[above=-2pt] {\box\xrat@above}
		(\xrat@len,0) ;}}
\makeatother

\tikzset{Rightarrow/.style={double equal sign distance,>={Implies},->},
	triple/.style={-,preaction={draw,Rightarrow}}}

\begin{document}
\title{On fibred products of toposes}

\author{Léo Bartoli and Olivia Caramello}

\maketitle

\begin{abstract}
In the setting of relative topos theory, we show that the pullback of a relative presheaf topos on an arbitrary fibration is the relative presheaf topos on its inverse image. To this end, we develop and exploit a notion of extension with base change of a morphism of sites along the canonical functor. This provides a tool to compare the canonical relative site of the direct (resp. inverse) image with the direct (resp. inverse) image of the canonical relative site: although these operations do not commute in general, we show that in the case of the direct image they are related by an indexed weak geometric morphism, while in the case of the inverse image they can be compared via a cartesian functor that induces a suitable topology making them Morita-equivalent.
\end{abstract}

\tableofcontents

\section{Introduction}

Relative topos theory studies toposes over a fixed base topos. It formalizes the idea of \emph{doing topos theory in the world of the base topos}.

Just as absolute toposes can be presented by a category endowed with a topology, relative toposes can be presented by a fibration equipped with a suitable topology. Fibrations serve as the natural analogue of categories when working over a fixed base category, and to form a ``relative site'' (in the sense of \cite{CaramelloZanfa}) the given topology must contain the so-called Giraud topology, namely the minimal topology making the projection functor of the fibration a comorphism of sites. This topology represents the relative analogue of the trivial topology in the absolute case. In this article, we focus on ``relative presheaf toposes'', namely the relative toposes presented by a fibration endowed with its Giraud topology. 

One of the most important results in (absolute) topos theory is Diaconescu's equivalence. In the relative case, two main formulations are known: Diaconescu's version, expressed in the language of internal categories (see B.3.2 of \cite{elephant}), and Giraud's version, formulated in the setting of cartesian stacks (see \cite{giraud.classifying}) and generalized to arbitrary relative sites in \cite{bartolicaramello}. 

The theorems of Diaconescu and Giraud notably yield the existence and explicit description of pullbacks of certain relative presheaf toposes along arbitrary geometric morphisms: Diaconescu's version applies to internal categories, while Giraud's version to cartesian stacks. In both cases, the pullback is computed as the relative presheaf topos on the inverse image of the internal category or stack along the given geometric morphism. This suggests that the result may hold more generally \emph{for any fibration}. In this paper, we prove that this \emph{is} the case, namely, we show that the pullback of a relative presheaf topos on an arbitrary fibration (which is small relative to the base site) is the relative presheaf topos on its inverse image. Along the way to proving this main theorem, we establish several intermediate constructions and results of independent interest, such as the development of a notion of canonical extension with base change, which allows the computation of inverse images through Morita equivalences, and the study of the interaction between the inverse and direct images operations for fibrations and that of taking the canonical stack of a fibration.

To arrive at our main result, we address two crucial points allowing us to extend Giraud's method from the setting of cartesian stacks to that of arbitrary fibrations. The first, consisting in passing from fibrations to stacks, is easily dealt with: any fibration is Morita-equivalent to its stackification, i.e., the associated relative presheaf toposes are equivalent over the base topos. The second point, consisting in relating general fibrations to cartesian ones, is highly non-trivial, as already acknowledged by Giraud, who wrote in \cite{giraud.classifying}: 

\emph{As a by-product we get the existence of fibered products in the bicategory of topos. This result was first announced by M.~Hakim several years ago but was never published. I suspect that any written proof would have to deal with rather subtle technical difficulties about finite limits which are overcome here by the results of \S1.}

Giraud manages to overcome these difficulties \emph{in the cartesian setting} by using his version of relative Diaconescu's equivalence for cartesian stacks and showing that finite-limit-preserving morphisms of stacks - and hence morphisms of cartesian relative sites - are stable under transposition. However, his argument no longer works in the general case, since it crucially exploits the existence of finite limits at the site level. In fact, showing that morphisms of relative sites are stable under transposition in the general case is much more subtle as, unlike in the cartesian case, they are not characterized by a simple fibrewise condition such as the preservation of finite limits, but by more intricate filtering conditions (see Proposition 3.13 of \cite{bartolicaramello} for the details). Accordingly, we do not establish this fact by using such a concrete characterization - as that would also require a manageable description of the inverse image operation, which is not available in general - but by exploiting a more structural characterization, also established in \cite{bartolicaramello}: just as in the absolute case, the conditions for being a morphism of relative sites can be expressed as a finite-limit preservation requirement, though at a higher level of abstraction: that of its $\eta$-extension to the canonical stack of the corresponding relative topos - a construction introduced in our work \cite{bartolicaramello}. We face, however, a technical problem: the canonical stack construction and the $\eta$-extension operation \emph{do not commute} with the inverse image operation in general. This lack of commutativity at the site level does not prevent us, nonetheless, to obtain a commutativity at the topos level: this is achieved by introducing a canonical comparison functor and showing that, by endowing the source fibration with a topology naturally induced by it, the inverse image of the canonical stack and the canonical stack of the inverse image become Morita-equivalent. This identification enables the transposition of morphisms of sites along inverse and direct images, thereby allowing us to establish our claim and hence the desired result about bipullbacks of relative presheaf toposes.

More specifically, we can summarize the paper's contents as follows:

In section \ref{section2}, we review the basics of relative topos theory: the equivalent notions of fibration and of indexed category, and the crucial concept of relative site. A particular kind of relative sites will be the central object of study of this paper: fibrations endowed with their Giraud topology. This topology constitutes the relative analogue to the trivial topology, allowing to formulate the definition of \ac\ac relative presheaf toposes''. We also recall the notion of canonical relative site, which plays the same role as the canonical site of a topos in absolute topos theory. 

In section \ref{section3}, we develop the theory of $\eta$-extensions with a variable base. While in \cite{bartolicaramello} this construction was introduced over a fixed base site, the study of pullbacks of toposes requires us to move from triangles of geometric morphisms to squares, and hence to consider two different bases. We therefore extend the theory to the setting of two comorphisms of sites (presenting relative toposes) over distinct base sites related by a morphism of sites. The theory adapts well to this extended situation: we obtain results of the same nature as in the fixed-base case, including a characterization of morphisms of sites between two such comorphisms that induce commutative squares of geometric morphisms at the level of toposes, expressed in terms of fibrewise preservation of finite limits by the $\eta$-extension with base change. We also give a concrete characterization of these functors in the case of local fibrations (and hence also fibrations) as those morphisms of sites that are morphisms of local fibrations. 

As a byproduct of this analysis, we also uncover a comparison indexed weak geometric morphism relating the direct image of the canonical stack and the canonical stack of the direct image: this is nothing else than the $\eta$-extension of the direct image of the canonical functor $\eta$.

Section \ref{section4} is the core of the paper, where we tackle the construction of pullbacks of relative presheaf toposes presented by arbitrary fibrations. 

In subsection \ref{subsect1}, we recall the main steps of Giraud's proof in the cartesian case; as explained earlier, this is the strategy we aim to generalize, by using the framework of $\eta$-extensions. 

In subsection \ref{subsec2} we set up our context and explain why there is no need to work with stacks or to compute inverse images explicitly: all stackification procedures can be omitted thanks to the relative Morita-equivalence between a fibration and its stackification (i.e. the equivalence between the associated relative presheaf toposes). 

In subsections \ref{subsec3} and \ref{subsec4}, we encounter a first difficulty that Giraud circumvented by restricting to the cartesian case. In his framework, one can directly obtain a canonical morphism of sites between a fibration and its inverse image without explicitly computing the latter: this induces the needed projection geometric morphism of the bipullback square of toposes. This relies on the fact that morphisms of sites can be described as cartesian morphisms of fibrations, and cartesianness can be checked by using structural properties of the inverse image, rather than by computing it concretely. In our situation, by contrast, showing that the canonical morphism of fibrations between a fibration and its inverse image is itself a morphism of sites would require verifying filteredness conditions, which are extremely difficult to handle when inverse images are defined through morphisms of sites (for instance, for the inverse image of a geometric morphism). 

Working with comorphisms of sites provides a way around this obstacle. In that setting, one can easily construct a canonical functor between a fibration and its inverse image, which induces the relevant geometric morphism at the topos-level. Exploiting the duality between comorphisms and morphisms of sites established in \cite{denseness}, together with the invariance of relative presheaf toposes proved in subsection \ref{subsec3}, we deduce in subsection \ref{subsec4} the existence of a canonical geometric morphism between the relative presheaf topos on a fibration and that on its inverse image, inducing a commutative square of geometric morphisms at the topos level. This provides the first part of the needed equivalence: morphisms of fibrations which are also morphisms of sites are transposed, in a first direction, not only to morphisms of fibrations but to morphisms of sites as well. This is the first step towards proving that the relative presheaf topos on the inverse image yields a bipullback square, as desired. 

In subsection \ref{subsec5}, we establish the converse direction of this equivalence, namely that transposition in the second direction also restricts well to those morphisms of fibrations that also are morphisms of sites. This requires comparing two constructions: the transposition of the $\eta$-extension and the $\eta$-extension of the transposition. From Giraud's cartesian case, we know that the transposition of the $\eta$-extension of a morphism of sites still preserves finite limits; what remains is to prove that the $\eta$-extension of the transposition does as well. To this end, we construct a canonical functor relating the canonical stack on the inverse image and the inverse image of the canonical stack. The key point is that this functor naturally induces a topology on the inverse image of the canonical stack, producing an equivalence at the topos level. This allows us to complete the proof that morphisms of sites are preserved under transposition along inverse and direct images.

In subsection \ref{subsec6}, we synthesize the results of the section and state our main theorem: the bipullback of a relative presheaf topos on an arbitrary fibration is the relative presheaf topos on its inverse image.

\section{Preliminaries}\label{section2}

This work is set in the context of relative topos theory, approached via stacks or, more generally, via fibrations following \cite{CaramelloZanfa}. In this section is recalled the essential background on which we shall rely: the notions of fibration and comorphism of sites, and the concept of relative site. The trivial relative sites (i.e. those given by the Giraud topology) give rise to the \ac\ac relative presheaf toposes'' will be our main object of study. We also recall the construction of the canonical relative site of a relative topos, and the associated canonical functor, and fix the notations used throughout the paper.

\subsection{Terminology and notation}

A first notation that will appear repeatedly concerns the topos of sheaves on a site. To emphasize the idea of the topos of sheaves on \((\mathcal{C}, J)\) as a kind of cocompletion of the category \(\mathcal{C}\) subject to relations imposed by \(J\), we adopt the light and suggestive notation \(\widehat{\mathcal{C}}_J\).

The functor denoted by \( l_J : \mathcal{C} \to \widehat{\mathcal{C}}_J \) is the \emph{canonical functor} associated with the site $(\cc,J)$, defined as the composite of the Yoneda embedding \( y_{\cc} : \mathcal{C} \to \widehat{\mathcal{C}} \) with the sheafification functor \( a_J : \widehat{\mathcal{C}} \to \widehat{\mathcal{C}}_J \); for the sake of lighter notation, we will occasionally omit the index \( J \).

There will be many different adjoint functors throughout the text, and we will use the following notation. For a pair of adjoint functors $G \dashv F$ and an arrow $u : G(d) \to c$, we will denote its transpose by $u^t : d \to F(c)$; in the converse direction, for an arrow $v : d \to F(c)$, we will also write $v^t : G(d) \to c$ for its transpose.

For two functors $F: \cd \to \cc$ and $G : \cd' \to \cc$, we denote as $(F/G)$ the comma category having for objects the triplets $(d,d',u:F(d)\to G(d'))$. Also, if one of the two functors, for example $F$, is the identity, we use the more concise notation $(\cd/G)$.

\subsection{Fibrations and relative sites}

In this paper, we will be concerned with \emph{relative toposes} and \emph{relative geometric morphisms} between them: 

\begin{defn}
\begin{enumerate}[(a)]
    \item A \emph{relative topos over some base topos ${\cal E}$} is a geometric morphism $f: {\cal F} \to {\cal E}$.

    \item A \emph{relative geometric morphism} between two relative toposes $g : [f] \to [f']$ is a geometric morphism $g : {\cal F} \to {\cal F}'$ such that the following triangle of geometric morphisms commutes (up to iso):
\[\begin{tikzcd}
	{{\cal F}} && {{\cal F}'} \\
	& {{\cal E}}
	\arrow["g", from=1-1, to=1-3]
	\arrow["f"', from=1-1, to=2-2]
	\arrow["{f'}", from=1-3, to=2-2]
\end{tikzcd}\]
\end{enumerate}
\end{defn}

A natural way to induce relative toposes is through \emph{fibrations}. We recall that a (Street) fibration is a particular kind of functor $p : {\cd} \to {\cc}$ satisfying a certain lifting property:

\begin{defn}
\begin{enumerate}[(a)]
\item For a functor $p : {\cd} \to {\cc}$, a \emph{cartesian arrow} is an arrow $f : d' \to d$ in $\cd$ such that: for each arrow $h : p(d'') \to p(d')$ in $\cc$ such that $p(f)\circ h = p(g)$ for $g : d'' \to d$, there exists a unique arrow $h' : d'' \to d'$ such that $h=p(h')$ and $fh'=g$.

\item A \emph{fibration} is a functor $p : {\cd} \to {\cc}$ such that for each arrow $f : c \to p(d)$ in $\cc$ there exists a cartesian arrow $\widehat{f} : d' \to d$ in $\cd$ with $\sigma : p(d') \simeq c$ such that $p(\widehat{f}) \simeq f\sigma$.
\end{enumerate}
\end{defn} 

As is well known (see, for instance, Corollary 2.2.6 of \cite{CaramelloZanfa} for a detailed proof for Street fibrations), fibrations over a given base category $\cc$ are equivalent to $\cc$-indexed categories. The fibration associated to an indexed category is given by the \emph{Grothendieck construction}, as defined:

\begin{prop}
A $\cc$-indexed category is a pseudofunctor $\mbc : \cc^{\mathbbm{op}} \to \mathbf{CAT}$. We can functorialy associate a fibration to every such $\cc$-indexed category as follows: the category $\cg(\mbc)$ has for objects the pairs $(x,c)$ where $c$ is an object of $\cc$ and $x$ is an object of $\mbc(c)$, and for arrows the pairs $(u,f):(x,c) \to (x',c')$ where $f : c\to c'$ is an arrow in $\cc$ and $u : x \to \mbc(f)(x')$ is an arrow in $\mbc(c)$.
\end{prop}

Accordingly, we shall often present a fibration $p : \cd \to \cc$ by means of the corresponding indexed category $p : \mathcal{G}(\mbc) \to \cc$, in order to have a more manageable description of it.

Fibrations naturally come with a notion of morphism between them:

\begin{defn}
Let  $p : \mathcal{G}(\mathbb C) \to \cc$ be a fibration over $\cc$,  $p' : \mathcal{G}(\mathbb C') \to \cc'$ be a fibration over $\cc'$, and $B: \cc \to \cc'$. A \emph{morphism of fibrations (with base change, when $B$ is not the identity functor)} is a functor $A : \mathcal{G}(\mathbb C) \to  \mathcal{G}(\mathbb C')$ sending cartesian arrows for $p$ to cartesian arrows for $p'$ and such that the following square commutes:

\[\begin{tikzcd}
	{\cg(\mbc)} & {\cg(\mbc')} \\
	\cc & {\cc'}
	\arrow["A", from=1-1, to=1-2]
	\arrow["p"', from=1-1, to=2-1]
	\arrow["\simeq"{description}, draw=none, from=1-1, to=2-2]
	\arrow["{p'}", from=1-2, to=2-2]
	\arrow["B"', from=2-1, to=2-2]
\end{tikzcd}\]
\end{defn}

\begin{remark}\label{remfiber}
Notice that, even if we do not ask for $A$ to preserve cartesian arrows, it still induces functors $A_c$ between the fibers, namely the $A_c : \mbc(c) \to \mbc'(B(c))$ defined on objects by sending an $x$ of $\mbc(c)$ to the object $A(x,c)$ of $\mbc'(B(c))$.
\end{remark}

When working over a fixed base, this notion of morphism for fibrations corresponds to the following notion of morphism for indexed categories:.

\begin{defn}
Let $\cc$ be a category and $\mbc$, $\mbc'$ two $\cc$-indexed categories. A morphism of indexed categories $A : \mbc \to \mbc'$ is the data of a functor $A_c : \mbc(c) \to \mbc(c')$ for each object $c$ of $\cc$ such that, for each arrow $f : c' \to c$ in $\cc$, the following square commutes:

\[\begin{tikzcd}
	{\mbc(c)} & {\mbc(c')} \\
	{\mbc'(c)} & {\mbc'(c')}
	\arrow["{\mbc(f)}", from=1-1, to=1-2]
	\arrow["{A_c}"', from=1-1, to=2-1]
	\arrow["{A_{c'}}", from=1-2, to=2-2]
	\arrow["{\mbc'(f)}"', from=2-1, to=2-2]
\end{tikzcd}\]
\end{defn}

Fibrations induce relative toposes via the notion of comorphism of sites. Recall that a functor $F : (\cc',J') \to (\cc,J)$ between two sites is said to be a \emph{comorphism of sites} when for every $J$-covering sieve $S$ on an object of the form $F(c')$, there exists a covering $S'$ for $J'$ on $c'$ such that $F(S') \subseteq S$. These functors are known to induce geometric morphisms, and hence relative toposes, covariantly. Indeed, a comorphism of sites $F : (\cc',J') \to (\cc,J)$ induces a geometric morphism denoted as $C_F : \widehat{\cc'}_{J'} \to \widehat{\cc}_J$ having for inverse image $C_F^* := a_{J'}(-\circ F^{\textup{op}}) : \widehat{\cc}_J \to \widehat{\cc'}_{J'}$ (see, for instance, subsection 3.3 of \cite{denseness} for more details).

As one can present any absolute topos (i.e. topos over $\mathbf{Set}$) as coming from a category endowed with a topology, one can present every relative topos by the means of a fibration endowed with some topology containing a minimal one. This topology, defined in the following definition-proposition, is called the \emph{Giraud topology} and constitutes the relative analogue of the trivial topology in the absolute case, in the sense that it is the \emph{minimal} topology making the projection functor into a comorphism of sites: 

\begin{defn}[Theorem 3.13 \cite{denseness}]\label{comorphinduitgeom}
For a fibration $p : \mathcal{G}(\mathbb C) \to \cc$ over a base site $(\cc,J)$, the \emph{Giraud topology} (denoted in this paper as $Gir_{\mathbb C}$) on $\mathcal{G}(\mathbb C)$ for the topology $J$ is defined by declaring as covering sieves those containing a family of cartesian arrows $((f_i,1): (\mathbb C(f_i)(x),c_i) \to (x,c))_i$ whose projections $(f_i)_i$ form a $J$-covering family. The data $p : (\cg(\mbc),Gir_{\mbc}) \to (\cc,J)$ will sometimes be called a \emph{trivial relative site}, in the sense that the Giraud topology is the smallest topology making $p$ a comorphism of sites towards $(\cc, J)$.
\end{defn}

\begin{remark}
We do not need, in this article, to make any explicit reference to the base topology in the notation $Gir_{\mathbb C}$, since no ambiguity about it will arise. In fact, as we will mainly work over the canonical site $(\ce, J^{\textup{\textup{can}}}_{\ce})$ of a topos $\ce$, including $J^{\textup{\textup{can}}}_{\ce}$ in the notation of the Giraud topology for a fibration over $\ce$ would result in an overload of subscripts and superscripts.
\end{remark}

In this paper we will deal with fibrations that are small relative to the base site, in the sense of the following

\begin{defn}\label{J-smallgen}
A fibration $p : \mathcal{G}(\mathbb C) \to \cc$ over a base site $(\cc,J)$ is said to be \emph{$J$-small} if the associated trivial relative site $(\mathcal{G}(\mathbb C),Gir_{\mbc})$ is small-generated.
\end{defn}

\begin{remark}\label{remJ-smallgen}
If $\mbc$ is a small indexed category, that is, taking values in the bicategory $\mathbf{Cat}$ of small categories and the base category $\cc$ is small, the corresponding fibration $p : \mathcal{G}(\mathbb C) \to \cc$  $(\cc,J)$ is $J$-small for any topology $J$ on $\cc$.
\end{remark}

\begin{defn}[Definition 8.2.1. \cite{CaramelloZanfa}]
A \emph{relative site} is a fibration $p : (\mathcal{G}(\mathbb C),K) \to (\cc,J)$ over a site such that $K$ contains the associated Giraud topology and $(\mathcal{G}(\mathbb C),K)$ is small-generated.
\end{defn}

The terminology \ac\ac relative site'' is justified by the fact that it induces a \emph{relative topos}, just as (absolute) sites induce (absolute) toposes. Indeed, the projection functor $p$ of a relative site $p: (\mathcal{G}(\mathbb C),K) \to (\cc,J)$ is a comorphism of sites, as $K$ contains the Giraud topology (the minimal one making $p$ a comorphism of sites): it induces the relative topos $C_p: \widehat{\cg(\mathbb C)}_K \to \widehat{\cc}_J$.

Just as presheaf toposes can be seen as categories of sheaves on a category endowed with its trivial topology, relative presheaf toposes are defined as relative toposes arising from fibrations equipped with their Giraud topology:

\begin{defn}[Definition 8.1.1 \cite{CaramelloZanfa}]
A \emph{relative presheaf topos} is a relative topos equivalent to one induced by a trivial relative site $p : (\cg(\mbc), Gir_{\mbc}) \to (\cc, J)$, that is, by a fibration endowed with its Giraud topology. The relative topos $p : \widehat{\cg(\mbc)}_{Gir_{\mbc}} \to \widehat{\cc}_J$ induced by such a trivial relative site will be denoted by $p : \mathbf{Gir}(\mbc) \to \widehat{\cc}_J$ for convenience of notation; such a relative topos is also called the \emph{Giraud topos} of the $\cc$-indexed category $\mbc$.
\end{defn}

In the case of trivial relative sites, the projection functor $p$ is not only a comorphism of sites, but also a continuous one. Continuous functors are defined as follows:

\begin{defn}[Proposition 4.8. \cite{denseness}]
Let $p : (\cd,K) \to (\cc,J)$ be a functor between two sites. We say that $p$ is continuous when $(-\circ p)$ restricts to categories of sheaves on the sites. Such a functor induces an adjoint pair at the topos-level: $\Sh(i)^* \dashv \Sh(i)_* : \widehat{\cc}_J \to \widehat{\cd}_K$ with $\Sh(i)^* \simeq a_Jlan_pi_K$ and $\Sh(i)_* \simeq (-\circ p)$.
\end{defn}

Continuous functors can be explicitly characterized as follows:

\begin{prop}[Proposition 4.13 \cite{denseness}]\label{cofinalitycond}
Let $A : (\mathcal{C}, J) \to (\mathcal{D}, K)$ be a functor between two sites. Then $A$ is continuous if and only if:

\begin{enumerate}[(i)]
    \item It is cover-preserving;
    \item for any $J$-covering sieve $S$ on an object $c$ and any commutative square of the form
\[
\begin{tikzcd}
d \arrow[r] \arrow[d] & A(c') \arrow[d, "A(f)"] \\
A(c'') \arrow[r, "A(g)"'] & A(c)
\end{tikzcd}
\]
where $f : c' \to c$ and $g : c'' \to c$ are arbitrary arrows of $S$, there is a $K$-covering family $(d_i \to d)_i$ such that for each $i$, the composites $d_i \to A(c')$ and $d_i \to A(c'')$ belong to the same connected component of the category $(d_i \downarrow A\pi_S)$ (where $\pi_S : \int S \to \cc$ is the obvious projection).
\end{enumerate}
\end{prop}

In the case of trivial relative sites, we have:

\begin{thm}[Theorem 4.44 and Corollary 4.47. \cite{denseness}]\label{trivialsitescont}
Let $ p : (\cg(\mbc),Gir_{\mbc}) \to (\cc,J)$ be a trivial relative site. The functor $p$ is continuous. Moreover, any morphism of fibrations over a base site is a continuous functor between their associated Giraud sites.
\end{thm}

When working with continuous functors, the following result will be instrumental:

\begin{prop}\label{densereflectmorphsites}
Let $i : ({\cc'},J') \to ({\cc},J)$ be a dense morphism of sites, $F : ({\cc},J) \to ({\cd},K)$ a continuous functor and $F' : ({\cd'},K') \to ({\cc'},J')$ a functor. Then:

\begin{enumerate}[(i)]
    \item The functor $F'$ is continuous if and only if $iF'$ is.
    \item The functor $F$ is a morphism of sites if and only if $Fi$ is.
    \item The functor $F'$ is a morphism of sites if and only if $iF'$ is
\end{enumerate}
\end{prop}

\begin{proof}
For (i), we use the equivalent definition of continuous functors stated in Proposition 4.8 of \cite{denseness}: $F'$ is continuous if and only if precomposition with it preserves sheaves. Since $i$ is a morphism of sites inducing an equivalence of toposes, precomposition by $i$ is an equivalence of categories. Therefore, for any sheaf $G'$ on $(\mathcal{C}', J')$, there exists a sheaf $G$ on $(\mathcal{C}, J)$ such that $G \circ i \simeq G'$. It follows that $G' \circ F' \simeq G \circ (i \circ F')$, which is a sheaf by continuity of $i \circ F'$; thus, $F'$ is continuous.

For (ii) and (iii), we know that $F$ and $F'$ are continuous, by hypothesis for $F$ and by (i) for $F'$. Hence, the morphism of sites condition is reduced to the preservation of finite limits at the topos-level: this is given by the fact that, $\Sh(i)^*$ being an equivalence, it reflects finite-limit preserving functors.
\end{proof}

In order to study relative presheaf toposes over different base toposes, we will need base-change operations on the fibrations presenting them. A detailed account of these operations can be found in Chapter 3 of \cite{CaramelloZanfa}.

\begin{defn}\label{defimdirecte}
Let $F : \cc \to \cd$ be a functor. We define the bifunctor $(-\circ F^{\mathrm{op}}) : Ind_{\cd} \to Ind_{\cc}$ by sending any $\cd$-indexed category $\mbd$ to the $\cc$-indexed category $\mbd \circ F^{\mathrm{op}}$. Equivalently, on fibrations this functor acts by sending a fibration $p : \cg(\mbd) \to \cd$ to its bipullback along $F$, as shown below:

\[\begin{tikzcd}
	{\cg(\mbd\circ F^{\mathrm{op}})} & {\cg(\mbd)} \\
	{\cc} & \cd
	\arrow["{q^F_{\mbd}}", from=1-1, to=1-2]
	\arrow["{p'}"', from=1-1, to=2-1]
	\arrow["\lrcorner"{anchor=center, pos=0.125}, draw=none, from=1-1, to=2-2]
	\arrow["p", from=1-2, to=2-2]
	\arrow["F"', from=2-1, to=2-2]
\end{tikzcd}\]

Throughout the article, we will often denote this projection functor by $q^F_{\mbd}$.
\end{defn}

When $F$ is a morphism of sites, this base change corresponds to the \emph{direct image} operation. When $F$ is a comorphism of sites, the base change corresponds to the \emph{inverse image} (up to stackification, which will not be needed in our study, see \ref{subsec2}). 

The following proposition will be instrumental when working with direct images of fibrations:

\begin{prop}\label{projectionreflectcart}
Let $F : \cc \to \cd$ be a functor and $\mbd$ a $\cd$-indexed category. The projection functor $q^F_{\mbd} : \cg(\mbd F) \to \cg(\mbd)$ reflects cartesian arrows, i.e. it sends an arrow $f$ to a cartesian one if and only if $f$ was arleady cartesian.   
\end{prop}

In the setting of morphisms of sites (for example for $f^*$ when $f$ is a geometric morphism), the inverse image is induced by the Left Kan extension, as recalled:

\begin{prop}[Proposition 3.2.1 \cite{CaramelloZanfa}]\label{inverseimagepropdef}
Denote by $Ind^{s}_{\mathcal{C}}$ the sub-2-category of $Ind_{\mathcal{C}}$ of pseudofunctors with values in $\Cat$ (i.e. ‘small’ $\mathcal{C}$-indexed categories). Consider any functor $F : \mathcal{C} \to \mathcal{D}$ and the direct image 2-functor
\[
(- \circ F^{\mathrm{op}}) : Ind^{s}_{\mathcal{C}} \to Ind^{s}_{\mathcal{D}}
\]
which acts by precomposition with $F^{\mathrm{op}}$. The 2-functor $(-\circ F^{\mathrm{op}})$ has a left 2-adjoint, denoted by $Lan_{F^{\mathrm{op}}}$ which acts as follows: for any object $D$ of $\mathcal{D}$ denote by $\pi_F^D : (D / F) \to \mathcal{C}$ the canonical projection functor; then for $\mathbb{E} : \mathcal{C}^{\mathrm{op}} \to \Cat$, its \emph{inverse image} $Lan_{F^{\mathrm{op}}}(\mathbb{E}) : \mathcal{D}^{\mathrm{op}} \to \Cat$ is defined componentwise as
    \[
    Lan_{F^{\mathrm{op}}}(\mathbb{E})(D) := \operatorname{colim}_{\mathrm{ps}} \left( ({D} / F)^{\mathrm{op}} \xrightarrow{(\pi_F^D)^{\mathrm{op}}} \mathcal{C}^{\mathrm{op}} \xrightarrow{\mathbb{E}} \Cat \right)
    \]

\noindent
The pseudofunctor $Lan_{F^{\mathrm{op}}}(\mathbb{E})$ is called the \emph{inverse image of $\mathbb{E}$ along $F$}.
\end{prop}

The computation of inverse image functors along morphisms of sites relies on this notion of Left Kan extension, which involves taking pseudocolimits of categories  -  a process that is notoriously heavy and often impractical. To circumvent this difficulty, we will exploit the abstract duality developed in \cite{denseness}, which allows us to replace morphisms of sites by comorphisms. As previously mentioned, inverse images are much easier to handle in the comorphism setting, since they correspond to pullbacks of fibrations.

Although the expression of inverse images is rather intricate when working with a general functor, it significantly simplifies when the functor admits a left adjoint:

\begin{prop}\label{adjointsinversedirectimage}
Let $G \dashv F : \cc \to \cd$ be a pair of adjoint functors, and $\mathbb C$ a $\cc$-indexed category. We have that $Lan_{F^{\mathrm{op}}}(\mbd) \simeq \mbd \circ G^{\mathrm{op}}$.
\end{prop}

\begin{proof}
It is a classical fact that adjunctions are preserved by any $2$-functor. In our setting, the adjunction $G \dashv F$ is preserved by the $2$-functor that assigns to each category the bicategory of indexed categories over it, and to each functor the operation of precomposition.
\end{proof}

\subsection{Canonical relative site, canonical functor}\label{canonicalresitecanonicalfunct}
We have just seen that every relative site induces a relative topos; conversely, every relative topos arises from a relative site. More specifically, this is achieved by means of the the so-called \emph{canonical relative site} of a relative topos, presented in the following definition/proposition:

\begin{defn}[section 8.2.2 of \cite{CaramelloZanfa}]\label{canstackresume}
Let $f : {\cal F} \to {\cal E}$ be a relative topos.

\begin{enumerate}[(a)]
    \item The projection $\pi_f:({\cal F}/f^* ) \to {\cal E}$ is a fibration (and even a stack for the canonical topology on ${\cal E}$).
    \item Its associated indexed category will be denoted $S_f : {\cal E}^{\textup{op}} \to \mathbf{CAT}$ and acts by sending an object $E$ of ${\cal E}$ to the slice topos $({\cal F}/f^*(E))$, and an arrow $u : E \to E'$ to the pullback functor $S_f(u) : ({\cal F}/f^*(E')) \to ({\cal F}/f^*(E))$. If $\ce$ is presented as $\ce \simeq \widehat{\cc}_J$, then we have a version of this canonical stack on $(\cc,J)$, namely: $S_f\circ l_J$.
    \item In this light, a cartesian arrow in $({{\cal F}}/f^*)$ between $(F,E,\alpha: F \to f^*E)$ and  $(F',E',\alpha': F' \to f^*E')$ is the data of a pair of arrows $(g:F \to F', g':E\to E')$ such that the commutative square given by $\alpha' \circ g = f^*(g') \circ \alpha$ is a pullback one:

    \[\begin{tikzcd}
	F & {F'} \\
	{f^*E} & {f^*E'}
	\arrow["g", from=1-1, to=1-2]
	\arrow["\alpha"', from=1-1, to=2-1]
	\arrow["\lrcorner"{anchor=center, pos=0.125}, draw=none, from=1-1, to=2-2]
	\arrow["{\alpha'}", from=1-2, to=2-2]
	\arrow["{f^*(g')}"', from=2-1, to=2-2]
    \end{tikzcd}\]
    \item The indexed category $S_f : \ce^{\textup{op}} \to \mathbf{CAT}$ associated with a relative topos $f: \cf \to \ce$ is precisely the direct image of the indexed category $S_{1_{\cf}} : \cf^{\textup{op}} \to \mathbf{CAT}$ along the geometric morphism $f$ (so is its associated fibration). We will denote the canonical stack of the identity geometric morphism $1_{\ce} : \ce \to \ce$ by $S_{\cal E}$.
    
    \item The topology $J_f$ on the comma category $({\cal F}/f^*)$ has as covering sieves the sieves $S$ whose projection $\pi_{{\cal F}}(S)$ in $({\cal F},J_{{\cal F}}^{\textup{can}})$ is covering; that is, which are covering on the first component.

    \item  Together with the topology $J_f$, this construction provides a relative site $\pi_f:(({\cal F}/f^* ),J_f) \to ({\cal E},J_{{\cal E}}^{\textup{can}})$.

\end{enumerate} 
\end{defn}

\begin{remark}
Since our focus will be on relative presheaf toposes (i.e., relative toposes of the form $C_p : \mathbf{Gir}(\mbc) \to \ce$), we adopt a lighter notation for their canonical stacks: $S_{\mbc}$ instead of $S_{C_p}$. This not only avoids an overload of subscripts, but will also make the connection between inverse images and the canonical stack more transparent.
\end{remark}

This canonical relative site allows us to recover our initial relative topos:

\begin{thm}[Theorem 8.2.5. \cite{CaramelloZanfa}]\label{canonicalrelativesiteMoritaequi}
Let $f : {\cal F} \to {\cal E}$ be a relative topos. The functor $\pi_{{\cal F}} : (({\cal F}/f^* ),J_f) \to  ({\cal F},J_{{\cal F}}^{\textup{can}})$ is a dense morphism and a dense comorphism of sites, and we have the following equivalence of relative toposes over ${\cal E}$:

\[\begin{tikzcd}
	{{\cal F}} && {\widehat{{\cal F}/f^*}_{J_f}} \\
	& {{\cal E}}
	\arrow["{\Sh(\pi_{{\cal F}})}", shift left=4, from=1-1, to=1-3]
	\arrow["\simeq"{description}, shift left=2, draw=none, from=1-1, to=1-3]
	\arrow["f"', from=1-1, to=2-2]
	\arrow["{C_{\pi_{{\cal F}}}}", shift left, from=1-3, to=1-1]
	\arrow["{C_{\pi_f}}", from=1-3, to=2-2]
\end{tikzcd}\]
\end{thm}

Note that, when the base topos ${\cal E}$ is presented as $\widehat{\cc}_J$, all the previous definitions and properties can be stated not only over ${\cal E}$, but also over $(\cc,J)$; see subsections 2.2 and 2.3 of \cite{bartolicaramello}. However, it can be more convenient to work over ${\cal E}$ rather than over an arbitrary base site since $\cal E$ has all finite limits, which makes the canonical stack $S_f$ a \emph{cartesian} fibration:

\begin{defn}[\cite{giraud.classifying}]
Let $\cc$ be a cartesian category. We say that $p : \cg(\mbc) \to \cc$ is a cartesian fibration if $\cg(\mbc)$ has finite limits and $p$ preserves them. Alternatively, from an indexed point of view, this is equivalent to $\mbc(c)$ having finite limits for every object $c$ of $\cc$, and the functors $\mbc(f)$ preserving them for every $f : c' \to c$ in $\cc$.
\end{defn}

Cartesian fibrations enjoy good stability properties under direct images, as established in the following proposition:

\begin{prop}\label{imdirectecartreflectlimfin}
Let $F : \cc' \to \cc$ be a cartesian functor between cartesian categories and $\mbc$ a $\cc$-indexed cartesian category. The $\cc'$-indexed category $\mbc F$ is cartesian and the projection functor $q^F_{\mbc} : \cg(\mbc F) \to \cg(\mbc)$ (cf. Definition \ref{defimdirecte}) reflects finite limits. 
\end{prop}
\begin{proof}
Straightforward computation from the characterization of limits in a fibration (see Proposition 2.4.2 \cite{CaramelloZanfa}).
\end{proof}

Now, the following is the definition of the relative analogue to the canonical functor denoted as $l$ in the absolute setting: it is a dense morphism and comorphism of sites having for domain the site and for codomain the canonical relative site. It exists in a full generality: for any comorphism of sites. 

\begin{defn}
Let $p: (\cd,K) \to (\cc,J)$ be a comorphism of sites. We denote by $\eta_{\cd} : (\cd,K) \to ((\widehat{\cd}_K/C_p^*),J_{C_p})$ the functor defined (via the Yoneda lemma) by sending an object $d$ of $\cd$ to $$(l_K(d),l_J(p(d)),a_K(ev_{1_{p(d)}}) : l_K(d) \to C_p^*l_J(p(d))),$$ and an arrow $f : d' \to d$ to $(l_K(f),l_J(p(f)))$.
\end{defn}

\begin{prop}[Proposition 2.5 \cite{fibered}]\label{etadense}
Let $p: (\cd,K) \to (\cc,J)$ be a comorphism of sites. The functor $\eta_{\cd} : (\cd,K) \to ((\widehat{\cd}_K/C_p^*),J_{C_p})$ is a dense morphism and a dense comorphism of sites. 
\end{prop}

As in the absolute setting, where the index $J$ of the functor $l_J$ is often omitted, we will likewise write $\eta_{\cd}$ simply as $\eta$ for the sake of simplicity.

When $p : (\cg(\mbc),K) \to (\cc,J)$ is a relative site, we denote $\eta_p$ as $\eta_{\mbc}$; in this case, this canonical dense functor has the additional property of preserving the fibration structure:

\begin{prop}[Proposition 3.14. \cite{bartolicaramello}]\label{etamorphfib}
Let $p : (\cg(\mbc),K) \to (\cc,J)$ be a relative site. The functor $\eta_{\mbc} : \cg(\mbc) \to (\widehat{\cg(\mbc)}_K/C_p^*)$ is a morphism of fibrations.
\end{prop}

Finally, the canonical stack allows to formulate the following result in the case of relative presheaf toposes:

\begin{cor}[Corollary 5.4.1 \cite{CaramelloZanfa}]
Consider a $\mathcal{C}$-indexed category $\mathbb{C}$ over a site $(\cc,J)$, and its opposite $\mathcal{C}$-indexed category $\mathbb{C}^V$, given by its composite with $(-)^{\mathrm{\textup{op}}}$ (that is, obtained by taking the opposite category fibrewise); then
\[
\mathbf{Gir}(\mathbb{C})
\simeq \mathrm{Ind}_{\mathcal{C}}(\mathbb{C}^V, \mathcal{S}_{{\widehat{\cc}_J}}l_J).
\] 
\end{cor}

This perspective allows us to view sheaves on the relative trivial site $(\cg(\mbc), Gir_{\mbc})$ as \emph{presheaves internal to the world of the base topos}. Indeed, passing from the ordinary to the relative setting, categories are replaced by indexed categories/fibrations, functors by morphisms of fibrations, and $\mathbf{Set}$ by the canonical stack of the base topos.

\section{Canonical extension with base change}\label{section3}

In \cite{bartolicaramello} we developed the theory of the so-called $\eta$-extension of continuous functors over a base site. This construction can be seen as a relative analogue of the canonical (left Kan) extension along the functors $l_J : \cc \to \widehat{\cc}_J$, and it played a central role in characterizing the appropriate analogues of morphisms of sites and continuous functors in the relative setting. In this section, we define the operation of $\eta$-extension with a variable base and discuss its main properties. It essentially works in the same way as over a fixed base, and it notably allows to understand, at an abstract level, which are the continuous functors between relative sites (over different bases related by a morphism of sites) which induce commutative squares of geometric morphisms. 

\begin{defn}\label{defnetaextbasechange}
Let $p : (\cd,K) \to (\cc,J)$, $p' : (\cd',K') \to (\cc',J')$ be two comorphism of sites, $B : (\cc',J') \to (\cc,J)$ a morphism of sites and $A : (\cd',K') \to (\cd,K)$ a continuous functor together with a natural transformation $\phi: pA\Rightarrow Bp'$, as depicted in the following diagram:

\[\begin{tikzcd}
	{(\cd',K')} & {(\cd,K)} \\
	{(\cc',J')} & {(\cc,J)}
	\arrow["A", from=1-1, to=1-2]
	\arrow["{p'}"', from=1-1, to=2-1]
	\arrow["\phi"{description}, Rightarrow, from=1-2, to=2-1]
	\arrow["p", from=1-2, to=2-2]
	\arrow["B"', from=2-1, to=2-2]
\end{tikzcd}\]
The \emph{$\eta$-extension of $(A,\phi)$} is the functor $$\widetilde{A}^*: (\widehat{\cd'}_{K'}/C_{p'}^*) \to (\widehat{\cd}_K/C_p^*)$$ sending an object $(F',E',\alpha': F' \to C_{p'}^*(E'))$ of $(\widehat{\cd'}_{K'}/C_{p'}^*)$ to the object 
$$(\Sh(A)^*(F'),\Sh(B)^*(E'),\widetilde{\phi}_{E'}\Sh(A)^*(\alpha') : \Sh(A)^*(F') \to C_p^*(\Sh(B)^*(E'))$$ of $(\widehat{\cd}_K/C_p^*)$, where $\widetilde{\phi}: \Sh(A)^*C_p^* \Rightarrow C_{p'}^*\Sh(B)^*$ is the natural transformation induced by the mate of $(-\circ \phi^{\textup{op}}) : (-\circ p'^{\textup{op}})(- \circ {B}^{\textup{op}}) \Rightarrow (-\circ A^{\textup{op}})(- \circ {p}^{\textup{op}})$.
\end{defn}

The following proposition presents the main properties of this operation: 

\begin{prop}\label{etaextensionbasechangeproperties}
Let $p : (\cd,K) \to (\cc,J)$, $p' : (\cd',K') \to (\cc',J')$ be two comorphism of sites, $B : (\cc',J') \to (\cc,J)$ a morphism of sites and $A : (\cd',K') \to (\cd,K)$ a continuous functor together with $\phi: pA\Rightarrow Bp'$ a natural transformation. Then:

\begin{enumerate}[(i)]
    \item The $\eta$-extension with base change 
    $$\widetilde{A}^* :(\widehat{\cd'}_{K'}/C_{p'}^*) \to (\widehat{\cd}_{K}/C_{p}^*)$$
    is a left adjoint. Its right adjoint is the functor 
    $$\widetilde{A}_* : (\widehat{\cd}_{K}/C_{p}^*) \to (\widehat{\cd'}_{K'}/C_{p'}^*)$$
    sending an object $(F,E,\alpha : F \to C_p^*(E))$ to the object defined by the following pullback square, where $\overline{\phi}$ is given by precomposition with $\phi$ followed by sheafification:

\[\begin{tikzcd}
	{\widetilde{A}_*(F,E,\alpha)} & {\Sh(A)_*(F)} \\
	{C_{p'}^*\Sh(B)_*(E)} & {\Sh(A)_*C_p^*(E)}
	\arrow[from=1-1, to=1-2]
	\arrow[from=1-1, to=2-1]
	\arrow["\lrcorner"{anchor=center, pos=0.125}, draw=none, from=1-1, to=2-2]
	\arrow["{\Sh(A)_*(\alpha)}", from=1-2, to=2-2]
	\arrow["{\overline{\phi}_E}"', from=2-1, to=2-2]
\end{tikzcd}\]

    \item The $\eta$-extension with base change 
    $$\widetilde{A}^* :((\widehat{\cd'}_{K'}/C_{p'}^*),J_{C_{p'}}) \to ((\widehat{\cd}_{K}/C_{p}^*),J_{C_{p}})$$
    is a continuous functor for the canonical relative topologies. Moreover, if $A$ is a morphism of sites, $\widetilde{A}^*$ also is.
    
    \item The following is a commutative square of colimit-preserving functors:
\[\begin{tikzcd}
	{\widehat{(\widehat{\cal D}_K/C_p^*)}_{J_{C_p}}} & {\widehat{(\widehat{\cal D'}_{K'}/C_{p'}^*)}_{J_{C_{p'}}}} \\
	{\widehat{\cal D}_K} & {\widehat{\cal D'}_{K'}}
	\arrow["{\Sh(\pi_{\widehat{\cal D}_K})^*}"', shift right=2, draw=none, from=1-1, to=2-1]
	\arrow["\simeq"{description}, from=1-1, to=2-1]
	\arrow["{\Sh(\widetilde{A}^*)^*}"', from=1-2, to=1-1]
	\arrow["{\Sh(\pi_{\widehat{\cal D'}_{K'}})^*}", shift left=3, draw=none, from=1-2, to=2-2]
	\arrow["\simeq"{description}, shift right=3, from=1-2, to=2-2]
	\arrow["{\Sh(A)^*}", from=2-2, to=2-1]
\end{tikzcd}\]
In particular, $\widetilde{A}^*$ and $A$ induce, at the topos level, the same inverse image.
    \item We have a canonical natural transformation induced by $\phi$, as pictured: 
    \[\begin{tikzcd}
	{(\widehat{\cd}_K/C_p^*)} & {(\widehat{\cd'}_{K'}/C_{p'}^*)} \\
	{(\cd,K)} & {(\cd',K')}
	\arrow["{\widetilde{A}^*}"', from=1-2, to=1-1]
	\arrow["{\eta_{\cd'}}", from=2-1, to=1-1]
	\arrow["\Phi"{description}, Rightarrow, from=2-1, to=1-2]
	\arrow["{\eta_{\cd}}"', from=2-2, to=1-2]
	\arrow["A", from=2-2, to=2-1]
    \end{tikzcd}\]
    Moreover, if $\phi$ is an isomorphism, this canonical arrow $\Phi$ also is an isomorphism.
    
    \item The following square commutes:
\[\begin{tikzcd}
	{(\widehat{\cal D'}_{K'}/C_{p'}^*)} && {(\widehat{\cal D}_K/C_p^*)} \\
	{\widehat{\cc'}_{J'}} && {\widehat{\cc}_{J}}
	\arrow["{\widetilde{A}^*}", from=1-1, to=1-3]
	\arrow[""{name=0, anchor=center, inner sep=0}, "{\pi_{C_{p'}}}"', from=1-1, to=2-1]
	\arrow[""{name=1, anchor=center, inner sep=0}, "{\pi_{C_p}}", from=1-3, to=2-3]
	\arrow["{\Sh(B)^*}"', from=2-1, to=2-3]
	\arrow["\simeq"{description}, draw=none, from=0, to=1]
\end{tikzcd}\]
    \item The functor $\widetilde{A}^*$ induces functors on the fibers (as in Remark \ref{remfiber}):

    $$\widetilde{A}^*_E : \widehat{\cd'}_{K'}/C_{p'}^*(E') \to \widehat{\cd}_{K}/C_{p}^*(\Sh(B)^*(E'))$$

    \item When $A$ is a morphism of sites, the following conditions are equivalent:
    \begin{itemize}
        \item $\widetilde{\phi} : C_{p'}\Sh(A) \simeq \Sh(B)C_{p}$ is an isomorphism, that is, $(A,\phi)$ induces a morphism of relative toposes $[\Sh(B)C_{p}] \to [C_{p'}]$.
        \item $\widetilde{A}^*$ is a morphism of fibrations with base change (it sends cartesian arrows for $\pi_{C_{p'}}$ to cartesian ones for $\pi_{C_p}$).
        \item $\widetilde{A}^*$ preserves finite limits not only globally, but also fibrewise, that is, the $\widetilde{A}^*_E$'s are finite-limit preserving.
    \end{itemize}
\end{enumerate}
\end{prop}

\begin{proof}
The proof of the proposition is completely analogous to that of Proposition 3.4 \cite{bartolicaramello} and left to the reader. 
\end{proof}

\begin{remark}\label{remA_B}
In general, in the same situation as previously
\[\begin{tikzcd}
	{({\cal D'},K')} & {({\cal D},K)} \\
	{({\cal C}',J')} & {({\cal C},J)}
	\arrow["A", from=1-1, to=1-2]
	\arrow["{p'}"', from=1-1, to=2-1]
	\arrow["\phi"{description}, Rightarrow, from=1-2, to=2-1]
	\arrow["p", from=1-2, to=2-2]
	\arrow["B"', from=2-1, to=2-2]
\end{tikzcd}\]
the fifth point of the proposition shows that we can factorize the functor $\widetilde{A}^*$ as follows:

\[\begin{tikzcd}
	{(\widehat{\cd'}_{K'}/C_p'^*)} & {(\widehat{\cd}_{K}/C_p^*\Sh(B)^*)} & {(\widehat{\cd}_{K}/C_p^*)} \\
	& {\widehat{\cc'}_{J'}} & {\widehat{\cc}_{J}}
	\arrow["{\widetilde{A}^*_B}", from=1-1, to=1-2]
	\arrow["{\widetilde{A}^*}", bend left = 40, from=1-1, to=1-3]
	\arrow["{\pi_{C_{p'}}}"', from=1-1, to=2-2]
	\arrow["{\Sh(B)^*_{S_{C_p}}}"{pos=0.3}, shift left, from=1-2, to=1-3]
	\arrow["{\pi_{\Sh(B)C_p}}", from=1-2, to=2-2]
	\arrow["\lrcorner"{anchor=center, pos=0.000125}, draw=none, from=1-2, to=2-3]
	\arrow["{\pi_{C_p}}", from=1-3, to=2-3]
	\arrow["{\Sh(B)^*}"', from=2-2, to=2-3]
\end{tikzcd}\]

As we saw in Definition \ref{canstackresume}, the pullback of $(\widehat{\cd}_K/C_p^*)$ along $\Sh(B)^*$ coincides with the canonical stack of $\Sh(B)C_p$. Furthermore, the projection $\Sh(B)^*_{S_{C_p}}$ of this pullback reflects cartesian arrows (Proposition \ref{projectionreflectcart}). Consequently, $\widetilde{A}^*$ is a morphism of fibrations if and only if $\widetilde{A}^*_B$ is. For this reason, we will occasionally use the notation $\widetilde{A}^*$ to denote $\widetilde{A}_B^*$.
\end{remark}

As shown by the following result, the $\eta$-extension operation with base change is functorial:

\begin{prop}\label{functorialityetaext}
Let $p : (\cd,K) \to (\cc,J)$, $p' : (\cd',K') \to (\cc',J')$, $p'' : (\cd'',K'') \to (\cc'',J'')$ be comorphisms of sites, $B : (\cc',J') \to (\cc,J)$ and $B' : (\cc',J'') \to (\cc',J')$ morphisms of sites, $A : (\cd',K') \to (\cd,K)$ and $A': (\cd'',K'') \to (\cd',K')$ continuous functors, together with natural transformations $\phi: pA\Rightarrow Bp'$ and $\phi' : p'A' \Rightarrow B'p''$, as depicted:
\[\begin{tikzcd}
	{(\cd'',K'')} & {(\cd',K')} & {(\cd,K)} \\
	{({\cc}'',J'')} & {({\cc'},J')} & {({\cc},J)}
	\arrow["{A'}", from=1-1, to=1-2]
	\arrow["{p''}"', from=1-1, to=2-1]
	\arrow["A", from=1-2, to=1-3]
	\arrow["{\phi'}"{description}, Rightarrow, from=1-2, to=2-1]
	\arrow["{p'}"', from=1-2, to=2-2]
	\arrow["\phi"{description}, Rightarrow, from=1-3, to=2-2]
	\arrow["p", from=1-3, to=2-3]
	\arrow["{B'}"', from=2-1, to=2-2]
	\arrow["B"', from=2-2, to=2-3]
\end{tikzcd}\]

The following isomorphism holds, where $\widetilde{A}^*$ (resp. $\widetilde{A'}^*$,  $\widetilde{AA'}^*$) are the $\eta$-extensions of $(A,\phi)$ (resp. $(A',\phi')$, $(AA',(B\phi') \circ (\phi A'))$):
$$\widetilde{AA'}^* \simeq \widetilde{A}^*\circ \widetilde{A'}^*.$$
\end{prop}

\begin{proof}
Straightforward.
\end{proof}

The following proposition spells out the conditions for an arbitrary continuous functor to induce an indexed weak geometric morphism, in the case of relative toposes presented by comorphisms of sites. 

\begin{prop}\label{concreteconditionsgeneral}
Let $p : (\cd,K) \to (\cc,J)$, $p' : (\cd',K') \to (\cc',J')$ be two comorphism of sites, $B : (\cc',J') \to (\cc,J)$ a morphism of sites and $A : (\cd',K') \to (\cd,K)$ a continuous functor together with $\phi: pA\Rightarrow Bp'$ a natural transformation, as pictured:

\[\begin{tikzcd}
	{({\cal D'},K')} & {({\cal D},K)} \\
	{({\cal C}',J')} & {({\cal C},J)}
	\arrow["A", from=1-1, to=1-2]
	\arrow["{p'}"', from=1-1, to=2-1]
	\arrow["\phi"{description}, Rightarrow, from=1-2, to=2-1]
	\arrow["p", from=1-2, to=2-2]
	\arrow["B"', from=2-1, to=2-2]
\end{tikzcd}\]
Then the following conditions are equivalent:

\begin{enumerate}[(i)]
    \item The pair $(\widetilde{A}^* \dashv \widetilde{A}_*)$ constitutes an indexed weak geometric morphism between the canonical stacks $S_{C_{p'}}$ and $S_{\Sh(B)C_{p}}$ (i.e. a pair of \emph{indexed} adjoint functors).
    \item The functors on the fibers
    
    $$\widetilde{A}^*_{E'} : S_{C_{p'}}(E') \to S_{\Sh(B)C_{p}}(E')$$ (for $E'$ object of $\widehat{\cc'}_{J'}$) 
    constitute a morphism of indexed categories.
    
    \item For every arrow $f' : c'' \to c'$ in $\cc'$ and object $d'$ of $\cd'$ together with an arrow $u' : p(d') \to c'$, the two following conditions are satisfied:
\begin{itemize}
    \item For any pair $(u'' : p(d) \to B(c''),g : d \to A(d'))$ such that $B(f')u''=B(u')\phi_{d'}p(g)$, there exist a covering $(f_i : d_i \to d)_i$ for $K$ together with triplets $(x_i : d_i \to A(\overline{d_i'}),\overline{g_i'} : \overline{d_i'} \to d',\overline{u_i'} : p'(\overline{d_i'}) \to c'')$ with $p(g_i')u'=f'\overline{u_i'} $, such that $B(u'')p(f_i) = B(\overline{u_i'})\phi_{\overline{d_i'}}p(x_i)$ and $gf_i = A(\overline{g_i'})x_i$.
     \item For every two triplets $(x_1 : d \to A(\overline{d_1'}),g_1' : \overline{d_1'} \to d',\overline{u_1'} : p'(\overline{d_1'}) \to c'')$ and $(x_2 : d \to A(\overline{d_2'}),g_2' : \overline{d_2'} \to d',\overline{u_2'} : p(\overline{d_2'}) \to c'')$ such that $B(\overline{u_1'})\phi_{\overline{d_1'}}p(x_1) = B(\overline{u_2'})\phi_{\overline{d_2'}}p(x_2)$ and $A(\overline{g_1'})x_1 =A(\overline{g_2'})x_2$, there exist a covering $(f_i : d_i \to d)_i$ for $K$ such that, denoting by $P^{f'}_{(d',u')}$ the presheaf defined by the following pullback square
\[\begin{tikzcd}
	{P^{f'}_{(d',u')}} & {{\cal D}'(-,d')} \\
	{{\cal C'}(p'-,c'')} & {{\cal C'}(p'-,c')}
	\arrow[from=1-1, to=1-2]
	\arrow[from=1-1, to=2-1]
	\arrow["\lrcorner"{anchor=center, pos=0.125}, draw=none, from=1-1, to=2-2]
	\arrow["{ev_{u'}}", from=1-2, to=2-2]
	\arrow["{f' \circ -}"', from=2-1, to=2-2]
\end{tikzcd}\]
    we have that $(x_1f_i : d_i \to A(\overline{d_1'}),g_1' : \overline{d_1'} \to d',\overline{u_1'} : p'(\overline{d_1}') \to c'')$ and $(x_2f_i : d_i \to A(\overline{d_2'}),g_2' : \overline{d_2'} \to d',\overline{u_2'} : p(\overline{d_2}') \to c'')$ are in the same connected component of $(d_i / A\pi^{f'}_{(d',u')})$ with $\pi^{f'}_{(d',u')}$ being the projection functor $\int P^{f'}_{(d',u')} \to {\cd'}$.
\end{itemize}
    
\end{enumerate}
\end{prop}

\begin{proof}

It is immediate that (i) implies (ii), since both components of an indexed adjunction are, by definition, indexed functors. For (ii) $\Rightarrow$ (i), we know that $\widetilde{A}^*$ is a morphism of opfibrations (Proposition 6.2.6 \cite{locfib}), which ensures that $\widetilde{A}_*$ is automatically indexed. It then remains for $\widetilde{A}^*$ to be an indexed functor, so that the pair $(\widetilde{A}_*, \widetilde{A}^*)$ forms an adjunction between indexed categories.

The equivalence between the first and third conditions is entirely analogous to the fixed base case of Proposition 6.3.1 of \cite{locfib}; the proof is left to the reader.
\end{proof}

\subsection{Direct images and canonical stacks}

In the previous section, we observed that the $\eta$-extension of a continuous functor $A$ factors through the direct image of the canonical stack, and that it is a morphism of fibrations if and only if one of its factors (namely, $\widetilde{A}^*_B$ from Remark \ref{remA_B}) is itself a morphism of fibrations. When a relative topos is presented via a fibration rather than an arbitrary comorphism, a base change is already available at the level of sites. The direct image can be taken either before or after applying the canonical stack construction.

In this section, we investigate the relation between the direct image of the canonical stack of a relative presheaf topos and the canonical stack of the relative presheaf topos associated to the direct image of the fibration. We will see that, if it does not commute in general, there is a canonical comparison indexed weak geometric morphism between them.

In order to study the canonical stack of the direct image, the following proposition will be crucial:

\begin{prop}\label{pullbackcontinuousalongfib}
Let $B : ({\cc'},J') \to ({\cc},J)$ be a continuous functor, and $\mathbb C$ a $\cc$-indexed category. The projection functor $B_{\mathbb C}$ is also continuous for the respective Giraud topologies:

\[\begin{tikzcd}
	{(\cg(\mbc B),Gir_{\mbc B})} & {(\cg(\mbc),{Gir}_{\mbc})} \\
	{(\cc',J')} & {(\cc,J)}
	\arrow["{B_{\mbc}}", from=1-1, to=1-2]
	\arrow["{p'}"', from=1-1, to=2-1]
	\arrow["p", from=1-2, to=2-2]
	\arrow[""{name=0, anchor=center, inner sep=0}, "B"', from=2-1, to=2-2]
	\arrow["\lrcorner"{anchor=center, pos=0.000125}, draw=none, from=1-1, to=0]
\end{tikzcd}\]
\end{prop}

\begin{proof}
We refer to the concrete characterization of continuous functors in terms of cofinality conditions expressed by Proposition \ref{cofinalitycond}. 

If we have a covering $((1,v_i') : (\mathbb CB(v_i')(x),c_i') \to (x,c'))_i$ for the Giraud topology on $\mathcal{G}(\mathbb CB)$, it is sent to $((1,B(v_i')) : (\mathbb C(B(v_i'))(x),B(c_i')) \to (x,B(x')))_i$ by $B_{\mathbb C}$, which is a covering since it consists of a family of cartesian arrows $(1,B(v_i'))$ with $B(v_i')$ $J$-covering (since $B$ is cover-preserving).
For the second part of the cofinality conditions, let $((1,v_i') : (\mathbb CB(v_i')(x),c_i') \to (x,c'))_i$ be covering for $Gir_{\mathbb CB}$, and let a commutative square as below:

\[\begin{tikzcd}
	& {(x_{ij},c_{ij})} \\
	{B_{\mathbb C}(\mathbb CB(v_i)(x),c'_i)} && {B_{\mathbb C}(\mathbb CB(v_j)(x),c'_j)} \\
	& {B_{\mathbb C}(x,c')}
	\arrow["{(u_i,g_i)}"', from=1-2, to=2-1]
	\arrow["{(u_j,g_j)}", from=1-2, to=2-3]
	\arrow["{B_{\mathbb C}(1,v_i')}"', from=2-1, to=3-2]
	\arrow["{B_{\mathbb C}(1,v_j')}", from=2-3, to=3-2]
\end{tikzcd}\]

Since $B$ is continuous and the $v_i'$ are $J'$-covering, there exist a $J$-covering $(w_{ijk} : c_{ijk} \to c_{ij})_{k}$ and a connecting zig-zag as in the following diagram:

\[\begin{tikzcd}
	&&& {c_{ijk}} \\
	\\
	&&& {c_{ij}} \\
	{B(c_i')} && {B(c_1')} && {B(c_n')} && {B(c_j')} \\
	&&& {B(c')}
	\arrow["{w_{ijk}}"{description}, from=1-4, to=3-4]
	\arrow["{s_1}"{description}, from=1-4, to=4-3]
	\arrow["{s_n}"{description}, from=1-4, to=4-5]
	\arrow["{g_i}"', from=3-4, to=4-1]
	\arrow["{g_j}", from=3-4, to=4-7]
	\arrow["{B(v_i')}"', from=4-1, to=5-4]
	\arrow["{B(a_1)}"{description}, from=4-3, to=4-1]
	\arrow[squiggly, no head, from=4-3, to=4-5]
	\arrow["{B(v'_1)}"{description}, from=4-3, to=5-4]
	\arrow["{B(a_n)}"{description}, from=4-5, to=4-7]
	\arrow["{B(v'_n)}"{description}, from=4-5, to=5-4]
	\arrow["{B(v_j')}", from=4-7, to=5-4]
\end{tikzcd}\]

Then, we just need to lift this zig-zag to obtain one for our setting in $\mathcal{G}(\mathbb CB)$:

\begin{adjustbox}{scale=0.6}
$\begin{tikzcd}
	&& {(\mathbb C(w_{ijk})(x_{ij}),c_{ijk})} \\
	\\
	&& {(x_{ij},c_{ij})} \\
	{B_{\mathbb C}(\mathbb CB(v_i')(x),c'_i)} & {B_{\mathbb C}(\mathbb CB(v_1')(x),c'_1)} && {B_{\mathbb C}(\mathbb CB(v_n')(x),c'_n)} & {B_{\mathbb C}(\mathbb CB(v_j')(x),c'_j)} \\
	&& {B_{\mathbb C}(x,c')}
	\arrow["{(1,w_{ijkl})}"{description}, from=1-3, to=3-3]
	\arrow["{(\mathbb C(w_{ijk})(u_i),s_1)}"{description}, from=1-3, to=4-2]
	\arrow["{(\mathbb C(w_{ijk})(u_i),s_n)}"{description}, from=1-3, to=4-4]
	\arrow["{(u_i,g_i)}"', from=3-3, to=4-1]
	\arrow["{(u_j,g_j)}", from=3-3, to=4-5]
	\arrow["{B_{\mathbb C}(1,v_i')}"', from=4-1, to=5-3]
	\arrow["{B_{\mathbb C}(1,a_1)}", from=4-2, to=4-1]
	\arrow[squiggly, no head, from=4-2, to=4-4]
	\arrow["{B_{\mathbb C}(1,v'_1)}"{description}, from=4-2, to=5-3]
	\arrow["{B_{\mathbb C}(1,a_n)}"', from=4-4, to=4-5]
	\arrow["{B_{\mathbb C}(1,v'n)}"{description}, from=4-4, to=5-3]
	\arrow["{B_{\mathbb C}(1,v_j')}", from=4-5, to=5-3]
\end{tikzcd}$
\end{adjustbox}

\end{proof}

Let us now assume that $A$ is a morphism of fibrations between relative sites with their Giraud topologies:

\[\begin{tikzcd}
	{(\cg(\mbc'),Gir_{\mbc'})} & {(\cg(\mbc),Gir_{\mbc})} \\
	{(\cc',J')} & {(\cc,J)}
	\arrow["A", from=1-1, to=1-2]
	\arrow["{p'}"', from=1-1, to=2-1]
	\arrow["\simeq"{description}, draw=none, from=1-1, to=2-2]
	\arrow["p", from=1-2, to=2-2]
	\arrow["B"', from=2-1, to=2-2]
\end{tikzcd}\]

We can factorize $A$ as follows:

\begin{equation} \label{diag:pullback_site}
\begin{adjustbox}{center, max width=\linewidth}
\begin{tikzcd}
	{(\cg(\mbc'),Gir_{\mbc'})} & {(\cg(\mbc B),Gir_{\mbc B})} & {(\cg(\mbc),Gir_{\mbc})} \\
	& {(\cc',J')} & {(\cc,J)}
	\arrow["{A_B}", from=1-1, to=1-2]
	\arrow["A", bend left = 25, from=1-1, to=1-3]
	\arrow["{p'}"', from=1-1, to=2-2]
	\arrow["{B_{\mbc}}", from=1-2, to=1-3]
	\arrow[from=1-2, to=2-2]
	\arrow["p", from=1-3, to=2-3]
	\arrow[""{name=0, anchor=center, inner sep=0}, "B"', from=2-2, to=2-3]
	\arrow["\lrcorner"{anchor=center, pos=0.125}, draw=none, from=1-2, to=0]
\end{tikzcd}
\end{adjustbox}
\end{equation}
 
Since $A_B$ is a morphism of fibrations, it is continuous for the Giraud topologies, and the previous proposition ensures that $B_{\mbc}$ also is. 

These functors induce the following new factorization at the level of canonical stacks:

\[\begin{tikzcd}
	& {(\mathbf{Gir}(\mathbb CB)/C_{p'_B})} \\
	{(\mathbf{Gir}(\mathbb C')/C_{p'}^*)} & {(\mathbf{Gir}(\mathbb C)/C_{p}^*\Sh(B)^*)} & {(\mathbf{Gir}(\mathbb C)/C_{p}^*)} \\
	& {\widehat{\cc'}_{J'}} & {\widehat{\cc}_J}
	\arrow["{\nu_{\mathbb C}^B}", dashed, from=1-2, to=2-2]
	\arrow["{\widetilde{B_{\mathbb C}}^*}", from=1-2, to=2-3]
	\arrow["{\widetilde{A_B}^*}", from=2-1, to=1-2]
	\arrow["{\widetilde{A}^*_B}"', from=2-1, to=2-2]
	\arrow["{\pi_{C_{p'}}}"', bend right =10, from=2-1, to=3-2]
	\arrow["{\Sh(B)^*_{S_{\mathbb C}}}"', shift right, from=2-2, to=2-3]
	\arrow["{\pi_{\Sh(B)C_{p}}}"', from=2-2, to=3-2]
	\arrow["{\pi_{C_p}}", from=2-3, to=3-3]
	\arrow[""{name=0, anchor=center, inner sep=0}, "{\Sh(B)^*}"', from=3-2, to=3-3]
	\arrow["\lrcorner"{anchor=center, pos=0.125}, draw=none, from=2-2, to=0]
\end{tikzcd}\]

Indeed, $B_{\mbc}A_B \simeq A$, whence the functoriality of the $\eta$-extension (Proposition \ref{functorialityetaext}) yields $\widetilde{A}^* \simeq \widetilde{B_{\mbc}}^* \widetilde{A_B}^*$. The functor $\nu_{\mbc}^B$ comes from the universal property of the pullback.

Since $A$ is a morphism of fibrations, we know that $A_B$ is a morphism of fibrations; hence, $\widetilde{A_B}^*$ is a morphism of fibrations in virtue of Theorem 6.3.5. \cite{locfib}. Since $\Sh(B)^*_{S_{\mathbb C}}$ is a morphism of fibrations with base change, and we have that $\widetilde{A}^* \simeq \Sh(B)^*_{S_{\mathbb C}}\circ \nu_{\mathbb C}^B \circ \widetilde{A_B}^*$, there just remains to show that $\nu_{\mathbb C}^B$ is a morphism of fibrations if we want to obtain that $\widetilde{A}^*$ also is. 

The preceding discussion highlights the central role of the comparison functor $\nu^B_{\mbc}$, which relates the direct image of the canonical stack to the canonical stack of the direct image: 

\[\begin{tikzcd}
	{\mathcal{G}(S_{\mathbb C B})} && {\mathcal{G}(S_{\mathbb C}B)} && {\mathcal{G}(S_{\mathbb C})} \\
	& {\mathcal{G}(\mathbb C B)} && {\mathcal{G}(\mathbb C)} \\
	& {({\cc'},J')} && {({\cc},J)}
	\arrow["{\nu_{B}^{\mathbb C}}"', dashed, from=1-1, to=1-3]
	\arrow["{\widetilde{B_{\mathbb C}}}", bend left =30, from=1-1, to=1-5]
	\arrow["{\pi_B}"'{pos=0.7}, bend right = 10, from=1-1, to=3-2]
	\arrow["{B_{S_{\mathbb C}}}"{description}, from=1-3, to=1-5]
	\arrow["{\pi'}"{description, pos=0.7}, bend left = 10, from=1-3, to=3-2]
	\arrow["\lrcorner"{anchor=center, pos=0.000125}, draw=none, from=1-3, to=3-4]
	\arrow["\pi"{description, pos=0.7}, bend left = 10, from=1-5, to=3-4]
	\arrow["{\eta'}"{description}, from=2-2, to=1-1]
	\arrow["{\eta B}"{description}, from=2-2, to=1-3]
	\arrow["{B_{\mathbb C}}"{description}, from=2-2, to=2-4]
	\arrow["{p'}"', from=2-2, to=3-2]
	\arrow["\eta"{description}, from=2-4, to=1-5]
	\arrow["p"', from=2-4, to=3-4]
	\arrow[""{name=0, anchor=center, inner sep=0}, "B"', from=3-2, to=3-4]
	\arrow["\lrcorner"{anchor=center, pos=0.000125}, draw=none, from=2-2, to=0]
\end{tikzcd}\]

This indexed weak geometric morphism can be understood by the means of the indexed $2$-category 
    \[
    \textup{RelTrivSites}: \mathbf{Sites}^{\textup{op}} \to \mathbf{CAT}
    \]
which sends a site $({\cc},J)$ to the bicategory of fibrations over ${\cc}$ endowed with their Giraud topologies relative to $J$ and morphisms of fibrations between them; it acts on a morphism of sites $B : ({\cc'},J') \to ({\cc},J)$ by sending a fibration $\mathcal{G}(\mathbb C)$ (endowed with its Giraud topology relative to $J$) to the fibration $\mathcal{G}(\mathbb CB)$ (endowed with its Giraud topology relative to $J'$). 

Notice that we could restrict our attention to fibered categories over the site, since these are equivalent to relative trivial sites over the base site. However, performing base change along a morphism of sites (or more generally, along a continuous functor) naturally induces continuous functors for the associated Giraud topologies (cf. Proposition \ref{pullbackcontinuousalongfib}). For this reason, it is useful to regard such indexed categories as trivial relative sites.

For each site $({\cc},J)$, we have a $2$-functor

$$S^{({\cc},J)} : \textup{RelTrivSites}({\cc},J) \to \textup{RelTrivSites}({\cc},J)$$ 
sending a relative trivial site to its canonical stack, and a morphism of fibrations to its $\eta$-extension (which is a morphism of fibrations, see Theorem 6.3.5 \cite{locfib} ). This does not a priori constitute an indexed $2$-natural transformation from $\textup{RelTrivSites}$ to itself; indeed, in the following square:

\[\begin{tikzcd}
	{\textup{RelTrivSites}({\cc},J)} & {\textup{RelTrivSites}({\cc'},J')} \\
	{\textup{RelTrivSites}({\cc},J)} & {\textup{RelTrivSites}({\cc'},J')}
	\arrow["{S^{({\cc},J)}}", from=1-1, to=2-1]
	\arrow["{\nu^B}"{description}, Rightarrow, from=1-1, to=2-2]
	\arrow["{B_*}"', from=1-2, to=1-1]
	\arrow["{S^{({\cc'},J')}}", from=1-2, to=2-2]
	\arrow["{B_*}", from=2-2, to=2-1]
\end{tikzcd}\]

\noindent we only have the comparison arrow $\nu^B$ as mentioned just before, which does not a priori need to be an isomorphism. The fact that it is pointwise a morphism of fibrations is equivalent to $\widetilde{B_{\mathbb C'}}$ also being for each $\cc'$-indexed category $\mathbb C'$: this comes from the fact that $B_{S_{\mathbb C'}} \nu^B_{\mathbb C'} \simeq \widetilde{B_{\mathbb C'}}$ and $B_{S_{\mathbb C'}}$ reflects cartesian arrows(Proposition \ref{projectionreflectcart}).

\begin{remark}
Let $({\cc},J)$ and $({\cc'},J')$ be cartesian sites, $\mathbb C$ and $\mathbb C'$ cartesian indexed categories, a morphism of sites $B$ and a morphism of fibrations $A$, as pictured:

\[\begin{tikzcd}
	{\mathcal{G}(\mathbb C)} & {\mathcal{G}(\mathbb C')} \\
	{{\cc}} & {{\cc'}}
	\arrow["A", from=1-1, to=1-2]
	\arrow["p"', from=1-1, to=2-1]
	\arrow["{p'}", from=1-2, to=2-2]
	\arrow["B"', from=2-1, to=2-2]
\end{tikzcd}\]
We have the factorization of $A$ as $A \simeq A_B \circ B_{\mbc '}$ (as in \ref{diag:pullback_site}) and the functor $A$ preserves finite limits if and only if $A_B$ does. Indeed, one can easily check that $B_{\mbc '}$ preserves finite limits fibrewise, and postcomposition with morphisms of cartesian fibrations reflects the preservation of finite limits (cf. Proposition \ref{imdirectecartreflectlimfin}). Hence, $A$ preserving finite limits is equivalent to the preservation of finite limits by $A_B$.

In the particular case of $\nu_{\mathbb C'}^B$, the previous discussion yields that $\widetilde{B_{\mathbb C'}}$ being cartesian is equivalent to $\nu_{\mathbb C'}^B$ also being, as $B_{S_{\mathbb C'}} \nu^B_{\mathbb C'} \simeq \widetilde{B_{\mathbb C'}}$ and $B_{S_{\mathbb C'}}$ reflects the preservation of finite limits. 
\end{remark}

In order to obtain that $\nu_{\mathbb C'}^B$ is a morphism of fibrations, the following result will be useful. It is an adaptation of Theorem 6.3.2 \cite{locfib} to the case of local fibrations (see Definition 3.2.1 \cite{locfib}) with base change: it gives a characterization of the continuous functors between local fibrations (over different base sites) which induce morphisms of fibrations with base change at the level of the associated canonical stacks. Since fibrations are particular cases of local fibrations, the result also notably applies to fibrations.

\begin{thm}\label{charactgeomtrelmorphfib}
Let $p$ and $p'$ be local fibrations, $B$ a morphism of sites and $A$ a continuous functor making the following square commutative (up to isomorphism):
\[\begin{tikzcd}
	{({\cd},K)} & {({\cd}',K')} \\
	{({\cc},J)} & {({\cc}',J')}
	\arrow["A", from=1-1, to=1-2]
	\arrow["p"', from=1-1, to=2-1]
	\arrow["\simeq"{description}, draw=none, from=1-1, to=2-2]
	\arrow["{p'}", from=1-2, to=2-2]
	\arrow["B"', from=2-1, to=2-2]
\end{tikzcd}\]
Then the following are equivalent:

\begin{enumerate}[(i)]
    \item The continuous functor $A$ is a morphism of local fibrations with base change (i.e. it sends locally cartesian arrows for $p$ on locally cartesian arrows for $p'$).
    \item The functor $\widetilde{A}_B^* : \mathcal{G}(S_{C_p}) \to \mathcal{G}(S_{\Sh(B)C_{p'}})$ is a morphism of fibrations and is the indexed left adjoint of $\widetilde{A}_*^B$, which is the obvious factorization of $\widetilde{A}_*$ through $\cg(S_{\Sh(B)C_{p'}})$; that is, they constitute together an indexed weak geometric morphism $\widetilde{A}_B : S_{\Sh(B)C_{p'}} \to S_{C_p}$.
    \item The functor $\widetilde{A}^* : \mathcal{G}(S_{C_p}) \to \mathcal{G}(S_{C_{p'}})$ is a morphism of fibrations with base change.
\end{enumerate}
\noindent
Moreover, for $A$ a morphism of sites, the following conditions are equivalent:

\begin{enumerate}[(i)]
    \item $A$ is a morphism of local fibrations with base change (it sends locally cartesian arrows for $p$ on locally cartesian arrows for $p'$).
    \item  $\widetilde{A}_B^* : \mathcal{G}(S_{C_p}) \to \mathcal{G}(S_{\Sh(B)C_{p'}})$ is a morphism of cartesian fibrations.
    \item  $\widetilde{A}^* : \mathcal{G}(S_{C_p}) \to \mathcal{G}(S_{C_{p'}})$ is a morphism of cartesian fibrations with base change.
    \item  $\Sh(A)$ is not only a geometric morphism, but also a relative one $$\Sh(A) : [\Sh(B)C_{p'}] \to [C_p].$$
\end{enumerate}
\end{thm}

\begin{proof}
The equivalence between conditions (iii) and (ii) in the first part of the theorem follows easily by exploiting the factorization $\widetilde{A}^*  \simeq B_{S_{C_{p'}}}\widetilde{A}_B^*$ and the fact that $B_{S_{C_{p'}}}$ reflects cartesian arrows (Proposition \ref{projectionreflectcart}), while the equivalence between conditions (i) and (iii) can be proved in the same way as for the case without base change, by adjusting the proof using the characterization of $B$ as a cover-preserving filtering functor to show the conditions of Proposition \ref{concreteconditionsgeneral} (iii).

The second part of the theorem is an instantiation of the first in the case of $A$ being a morphism of sites. In this setting, one can exploit the fact that $B_{S_{C_{p'}}}$ reflects not only cartesian arrows but also finite limits, and the equivalence between indexed geometric morphisms and relative geometric morphisms (subsection 6.1 of \cite{locfib}). Note that this is not the case for indexed weak geometric morphisms and relative weak geometric morphisms (cf. subsection 6.1 of \cite{locfib}).
\end{proof}

\begin{cor}
Let $p : (\cg(\mbc),K) \to (\cc,J)$ and $p' : (\cg(\mbc '),K') \to (\cc',J')$ two relative sites with $B : (\cc,J) \to (\cc',J')$ a morphism of sites, and $A: (\cg(\mbc),K) \to (\cg(\mbc '),K')$ a morphism of fibrations. We have:

\begin{enumerate}[(i)]
    \item If $A$ is continuous then $(\widetilde{A}^*_B \dashv \widetilde{A}_*^B)$ is an indexed weak geometric morphism $\widetilde{A}_B : S_{\Sh(B)C_{p'}} \to S_{C_p}$.
    \item If $A$ is a morphism of sites then $\Sh(A) : [\Sh(B)C_{p'}] \to [C_p]$ is a relative geometric morphism.
\end{enumerate}
\end{cor}

\begin{proof}
As fibrations are local fibrations, and continuous morphisms of fibrations are morphisms of local fibrations (Proposition 4.4.3 \cite{locfib}), the statements of the corollary follow as direct consequences of the previous theorem.
\end{proof}

\begin{remark}
A full concrete characterization of the lax morphisms of sites between arbitrary comorphisms of sites with base change inducing squares of commuting geometric morphisms has been given in section 2 of \cite{CaramelloOsmond}.
\end{remark}

Since fibrations are local fibrations and the projection functor from a pullback is a morphism of fibrations with base change (cf. Proposition \ref{projectionreflectcart}), the previous result allows us to obtain the following proposition:

\begin{prop}\label{comparisondirectimage}
Let $p' : \cg(\mbc ') \to (\cc ',J')$ a fibration, $B : (\cc,J) \to (\cc ',J')$ a morphism of sites, and $B_{\mbc '}$ the projection functor in the following pullback diagram:
\[\begin{tikzcd}
	{(\mathcal{G}(\mathbb C'B),Gir_{\mathbb C'B})} & {(\mathcal{G}(\mathbb C'),Gir_{\mathbb C'})} \\
	{({\cc},J)} & {({\cc}',J')}
	\arrow["{B_{\mathbb C'}}", from=1-1, to=1-2]
	\arrow["p"', from=1-1, to=2-1]
	\arrow["\simeq"{description}, draw=none, from=1-1, to=2-2]
	\arrow["\lrcorner"{anchor=center, pos=0.125}, draw=none, from=1-1, to=2-2]
	\arrow["{p'}", from=1-2, to=2-2]
	\arrow["B"', from=2-1, to=2-2]
\end{tikzcd}\]
Then the functors $\widetilde{B_{\mathbb C'}}$ and $\nu_{\mathbb C'}^B$ in the commutative diagram
\[\begin{tikzcd}
	{(\mathbf{Gir}(\mathbb C'B)/C_{p}^*l_{J})} \\
	& {(\mathbf{Gir}(\mathbb C')/C_{p'}^*l_{J'}B)} & {(\mathbf{Gir}(\mathbb C')/C_{p'}^*l_{J'})} \\
	& {({\cc},J)} & {({\cc}',J')}
	\arrow["{\nu^B_{\mathbb C'}}", dashed, from=1-1, to=2-2]
	\arrow["{\widetilde{B_{\mathbb C'}}}", bend left = 18, from=1-1, to=2-3]
	\arrow["\pi"', from=1-1, to=3-2]
	\arrow["{B_{S_{\mathbb C'}}}", from=2-2, to=2-3]
	\arrow["{\pi_B'}"', from=2-2, to=3-2]
	\arrow["{\pi'}", from=2-3, to=3-3]
	\arrow[""{name=0, anchor=center, inner sep=0}, "B"', from=3-2, to=3-3]
	\arrow["\lrcorner"{anchor=center, pos=0.125}, draw=none, from=2-2, to=0]
\end{tikzcd}\]
satisfy the following properties:

\begin{enumerate}[(i)]
    \item The functor $\widetilde{B_{\mathbb C'}}$ is a morphism of fibrations with base change.

    \item The functor $\nu_{\mathbb{C}'}^B$ acts by sending an object $(F,E,\alpha:F \to C_p^*l_J(E))$ of $(\mathbf{Gir}(\mbc'B)/C_p^*l_J)$ to the object $(\Sh(B_{\mbc'})^*(F),E,\psi_E\Sh(B_{\mbc'})^*(\alpha))$ of ${(\mathbf{Gir}(\mathbb C')/C_{p'}^*l_{J'}B)}$, where $\psi$ is the natural transformation $\Sh(B_{\mathbb C})^*C_{p}^* \Rightarrow  C_{p'}^*\Sh(B)^*$ induced by the natural isomorphism $p'B_{\mathbb C'}\simeq Bp$. 
    
    \item The functor $\nu_{\mathbb{C}'}^B$ is the $\eta$-extension of $\eta_{\mathbb C'} B$, that is ($\widetilde{\eta_{\mbc'}B}^*$), as described in the diagram: 

        \[\begin{tikzcd}
	{\mathcal{G}(S_{\mathbb C'}B)} && {\mathcal{G}(S_{\mathbb C'B})} \\
	& {\mathcal{G}(\mathbb C'B)}
	\arrow[""{name=0, anchor=center, inner sep=0}, "{\nu_{\mathbb C'}^B}"', from=1-3, to=1-1]
	\arrow["{\eta_{\mathbb C'}B}", from=2-2, to=1-1]
	\arrow["{\eta_{\mathbb C' B}}"', from=2-2, to=1-3]
	\arrow["\simeq"{description}, draw=none, from=2-2, to=0]
        \end{tikzcd}\]
    It is a morphism of fibrations and opfibrations (see subsection 6.1 of \cite{locfib}), and constitutes the inverse image part of an indexed weak geometric morphism:

    $$\widetilde{\eta_{\mbc'}B} : S_{\Sh(B)C_{p'}} \to S_{C_{p}}.$$

    \item Denoting by $\overline{\psi}_G$ the mate of the natural transformation $\Sh(B_{\mathbb C})^*C_{p}^* \Rightarrow  C_{p'}^*\Sh(B)^*$ induced by the natural isomorphism $p'B_{\mathbb C'}\simeq Bp$, the (indexed) direct image $\widetilde{\eta_{\mbc'}B}_*$ of $\widetilde{\eta_{\mbc'}B}$ acts by sending an object $(F',E,\alpha' : F' \to C_{p'}^*\Sh(B)^*(E'))$ of ${(\mathbf{Gir}(\mathbb C')/C_{p'}^*l_{J'}B)}$ to the object $(F,E,\alpha)$ of $(\mathbf{Gir}(\mbc'B)/C_p^*l_J)$ pictured in the following pullback diagram:

\[\begin{tikzcd}
	F & {C_{p}^*(E)} \\
	{\Sh(B_{\mathbb C})_*(F')} & {\Sh(B_{\mathbb C})_*C_{p'}^*\Sh(B)^*(E)}
	\arrow["\alpha", from=1-1, to=1-2]
	\arrow[from=1-1, to=2-1]
	\arrow["{\overline{\psi}_E}", from=1-2, to=2-2]
	\arrow[""{name=0, anchor=center, inner sep=0}, "{\Sh(B_{\mathbb C})_*(\alpha')}"', shift right, from=2-1, to=2-2]
	\arrow["\lrcorner"{anchor=center, pos=0.125}, draw=none, from=1-1, to=0]
\end{tikzcd}\]   
\end{enumerate}
\end{prop}

\begin{proof}
Since $B_{\mathbb{C}'}$ is a morphism of fibrations with base change, the previous theorem ensures that $\widetilde{B_{\mathbb{C}'}}$ is also a morphism of fibrations with base change, which proves point (i) of the proposition.

As the fibration $(\mathbf{Gir}(\mathbb C')/C_{p'}^*l_{J'}B)$ is a bipullback, any functor into it is (up to equivalence) determined by its composites with the associated projection functors. A direct computation shows that the functor described in point (ii) indeed corresponds to $\nu_{\mbc'}^B$ after applying these projection functors.

It can be readily verified that $\nu_{\mbc'}^B$ coincides with the $\eta$-extension of $\eta_{\mbc'}B$, by comparing the expression derived in point (ii) with the general construction of $\eta$-extensions. Moreover, since $\eta_{\mbc'}B$ is a continuous morphism of fibrations (see Proposition \ref{trivialsitescont}), the first point of the previous corollary ensures that $\nu_{\mbc'}^B \simeq \widetilde{\eta_{\mbc'}B}^*$ is the inverse image part of an indexed weak geometric morphism. In particular, $\nu_{\mbc'}^B$ is a morphism of opfibrations (see subsection 6.1 of \cite{locfib}). This completes the proof of point (iii).

Point (iv) follows directly from a straightforward computation, using the explicit description of $\widetilde{\eta_{\mbc'}B}^*$ given in point (ii).
\end{proof}

\begin{ex}
It is interesting to describe the behavior of the comparison indexed weak geometric morphism (Proposition \ref{comparisondirectimage}) in the case of an étale topos. Let $f : \cf \to \ce$ be a relative topos and $F$ an object of $\cf$; the projection $\pi_F : \cf/F \to \cf$ is a fibration, and the Giraud topology on it is the canonical topology on the category $\cf/F$ seen as a topos, whence its Giraud topos is itself. The canonical stack of the direct image of $\cf/F$ and the direct image of the canonical stack $S_{\cf/F}$ can be computed as follows:

\begin{itemize}
    \item The direct image of $S_{\cf/F}$ is the canonical stack $S_{f\pi_F}$, whence the direct image of $S_{\cf/F}$ is the canonical stack associated to the geometric morphism $\cf/F \to \cf \to \ce$.

    \item On the other hand, the direct image of the fibration $\cf/F$ along the morphism of sites $f^*$ is the fibration $\ce/f_*(F) \to \ce$. As previously mentioned, its Giraud topos is itself, so that the canonical stack of the direct image is the one associated to the geometric morphism $\ce /f_*(F) \to \ce$.
\end{itemize}

The inverse image of the comparison indexed weak geometric morphism $\widetilde{\eta_{F}f^*}$ between these two relative toposes is readily seen to be given, on the fiber over the terminal object $\mathbf{1}_{\ce}$, by:

\[\begin{tikzcd}[row sep=tiny]
	{\ce/f_*(F)} & {\cf/F} \\
	{(E,F,\alpha : E \to f_*(F))} & {(f^*(E),F,\alpha^t : f^*(E) \to F)}
	\arrow["{(\nu_{F}^{f^*})_{\mathbf{1}_{\ce}}}", from=1-1, to=1-2]
	\arrow[maps to, from=2-1, to=2-2]
\end{tikzcd}\]
\end{ex}

\begin{remark}
Whilst the previous proposition yields an indexed weak geometric morphism $( \widetilde{\eta_{\mathbb C'}B}^* \dashv \widetilde{\eta_{\mbc'}B}_* ) : S_{\Sh(B)C_{p'}} \to S_{C_p}$ where $\widetilde{\eta_{\mbc'}B}^* \simeq \nu_{\mbc'}^B$, such an adjunction is \emph{not}, in general, an indexed geometric morphism. Indeed, this phenomenon already appears in the case of étale toposes, as treated in the previous example. The terminal object of $\ce/f_*(F)$ is the identity morphism on $f_*(F)$, which is sent by $(\nu_{F}^{f^*})_{\mathbf{1}_{\ce}}$ to the unit $f_*f^*(F) \to F$ of the adjunction $(f^* \dashv f_*)$ at $F$; in general, this unit is not an isomorphism. As a consequence, $(\nu_{F}^{f^*})_{\mathbf{1}_{\ce}}$ does not preserve the terminal object, and the functor $(\nu_{F}^{f^*})$ is not an indexed geometric morphism.
\end{remark}

\section{Pullbacks of relative presheaf toposes}\label{section4}

\subsection{The cartesian case}\label{subsect1}

In this subsection, we present Giraud's computation (see \cite{giraud.classifying}) of the pullback of a presheaf topos on a \emph{cartesian stack}. Working in the cartesian setting makes it possible to formulate certain conditions - in particular those related to flatness  -  in a structural way already at the level of sites. Accordingly, his proof relies almost entirely on his cartesian version of Diaconescu's theorem. 

His result is the following: for a cartesian stack on a topos (with its canonical topology) $p: \cg (\mbc)\to \ce$ and a relative topos $f : \cf \to \ce$, the following square is a bipullback of toposes:

\[
\begin{tikzcd}
	{\mathbf{Gir}(f^*\mathbb C)} & {\mathbf{Gir}(\mbc)} \\
	{{{\cal F}}} & {{{\cal E}}}
	\arrow["{f^{\mathbb C}}", from=1-1, to=1-2]
	\arrow["{C_{p'}}"', from=1-1, to=2-1]
	\arrow["\lrcorner"{anchor=center, pos=0.000125}, draw=none, from=1-1, to=2-2]
	\arrow["{C_p}", from=1-2, to=2-2]
	\arrow["f"', from=2-1, to=2-2]
\end{tikzcd}
\]

As mentioned earlier, the key ingredient of Giraud's proof is its cartesian relative Diaconescu's theorem. Indeed, as in the absolute case, the cartesian setting allows for a convenient expression of flatness conditions, yielding a very manageable description of geometric morphisms at the level of sites: they are precisely the cartesian morphisms of fibrations. This is the content of the following theorem:

\begin{thm}[Proposition 2.4 \cite{giraud.classifying}]
Let $\ce$ be a topos, $f : \cf \to \ce$ a geometric morphism, and $p:\mathcal{G}(\mathbb C) \to \ce$ a cartesian stack. There is an equivalence of categories:
\[
\mathbf{FibCart}_{\ce}(\mathcal{G}(\mathbb C),({\cal F}/f^*))\simeq \mathbf{Top}/{\ce}([f],[C_p])
\]
\noindent where $\mathbf{FibCart}_{\ce}$ denotes the category of morphisms of cartesian fibrations over $\ce$.
\end{thm}

In order to deduce that $\mathbf{Gir}(g^*\mathbb C)$ is indeed the pullback of $\mathbf{Gir}(\mbc)$ when $\mathcal{G}(\mathbb C)$ is a cartesian stack, the proof proceeds as follows: given a geometric morphism $f$, the aim is to obtain an equivalence between geometric morphisms $h$ completing it in a commutative square, and relative geometric morphisms $\overline{h}$ as pictured:

\[\begin{tikzcd}
	{\cal G} \\
	& {\mathbf{Gir}(f^*\mathbb C)} & {\mathbf{Gir}(\mathbb C)} \\
	& {\cal F} & {{\cal E}}
	\arrow["{\overline{h}}"{description}, dashed, from=1-1, to=2-2]
	\arrow["h", bend left = 24, dashed, from=1-1, to=2-3]
	\arrow["g"', bend right = 18, from=1-1, to=3-2]
	\arrow["{f^{\mathbb C}}", from=2-2, to=2-3]
	\arrow["{C_{p'}}"', from=2-2, to=3-2]
	\arrow["\lrcorner"{anchor=center, pos=0.125}, draw=none, from=2-2, to=3-3]
	\arrow["{C_p}", from=2-3, to=3-3]
	\arrow["f"', from=3-2, to=3-3]
\end{tikzcd}\]

By the previous cartesian relative Diaconescu's theorem, relative geometric morphisms $\overline{h}$ are in correspondence with morphisms of cartesian fibrations $\overline{H}$ as pictured:

\[\begin{tikzcd}
	{(\cg/g^*)} \\
	& {\cg(f^*\mathbb C)} \\
	& {\cal F}
	\arrow["{\pi_g}"', bend right =16, from=1-1, to=3-2]
	\arrow["{\overline{H}}"{description}, dashed, from=2-2, to=1-1]
	\arrow["{p'}", from=2-2, to=3-2]
\end{tikzcd}\]

The morphism of cartesian fibrations $\overline{H} : \cg(f^*\mbc) \to (\cg/g^*)$ can be transposed into another morphism of fibrations $H : \cg(\mbc) \to f_*(\cg/g^*)$. Here, recall that $f_*(\cg/g^*) \simeq (\cf/(fg)^*)$ is the canonical stack of the geometric morphism $fg$ (see Definition \ref{canstackresume}). As finite limits and their preservation are computed fibrewise for fibrations, one can easily show that the transposition of general morphisms of fibrations restricts to morphisms of cartesian fibrations:

\[\begin{tikzcd}
	{\mathbf{Fib}_{\cal F}([p'],[\pi_g])} & {\mathbf{Fib}_{\cal E}([p],[\pi_{fg}])} \\
	{\mathbf{FibCart}_{\cal F}([p'],[\pi_g])} & {\mathbf{FibCart}_{\cal E}([p],[\pi_{fg}])}
	\arrow["\simeq"{description}, no head, from=1-1, to=1-2]
	\arrow["adjunction"{description}, shift left=3, draw=none, from=1-1, to=1-2]
	\arrow[hook, from=2-1, to=1-1]
	\arrow["restriction"{description}, shift left=3, draw=none, from=2-1, to=2-2]
	\arrow["\simeq"{description}, no head, from=2-1, to=2-2]
	\arrow[hook, from=2-2, to=1-2]
\end{tikzcd}\]

Hence, the cartesian morphisms of fibrations $\overline{H}$ described above are in equivalence with morphisms of cartesian fibrations $H$. A second application of the cartesian relative Diaconescu's theorem yields an equivalence between these morphisms of cartesian fibrations $H$ and the morphisms of relative toposes $h: [fg] \to [C_p]$, thus completing the proof.

All the steps can be visualized in the following diagram:
\[\begin{tikzcd}
	& {\mathbf{Fib}_{\cal F}([p'],[\pi_g])} & {\mathbf{Fib}_{\cal E}([p],[\pi_{fg}])} \\
	{step.2} & {\mathbf{FibCart}_{\cal F}([p'],[\pi_g])} & {\mathbf{FibCart}_{\cal E}([p],[\pi_{fg}])} & {step.3} \\
	{step.1} & {\mathbf{Geom}_{\cal F}([g],[C_{p'}])} & {\mathbf{Geom}_{\cal E}([fg],[C_{p}])} & {step.4}
	\arrow["\simeq"{description}, no head, from=1-2, to=1-3]
	\arrow["adjunction"{description}, shift left=3, draw=none, from=1-2, to=1-3]
	\arrow[bend left = 16, squiggly, maps to, from=2-1, to=2-4]
	\arrow[hook, from=2-2, to=1-2]
	\arrow["restriction"{description}, shift left=3, draw=none, from=2-2, to=2-3]
	\arrow["\simeq"', no head, from=2-2, to=2-3]
	\arrow["\simeq"{description}, no head, from=2-2, to=3-2]
	\arrow["{Diac. cart.}"', shift right=2, draw=none, from=2-2, to=3-2]
	\arrow[hook, from=2-3, to=1-3]
	\arrow["\simeq"{description}, no head, from=2-3, to=3-3]
	\arrow["{Diac. cart.}", shift left=3, draw=none, from=2-3, to=3-3]
	\arrow[squiggly, maps to, from=2-4, to=3-4]
	\arrow[squiggly, maps to, from=3-1, to=2-1]
	\arrow["\simeq"', no head, from=3-2, to=3-3]
\end{tikzcd}\]

\subsection{The general setting}\label{subsec2}

As commented by Giraud when establishing this result \cite{giraud.classifying}:

\emph{As a by-product [of the cartesian relative Diaconescu theorem], we obtain the existence of fibered products in the bicategory of toposes. This result was first announced by M. Hakim several years ago but was never published. I suspect that any written proof would have to address rather subtle technical difficulties regarding finite limits, which are bypassed here thanks to the results of [transposition of cartesian morphisms of fibrations]}

Indeed, Giraud's proof works thanks to the good behavior of finite limits and their preservation under transposition along direct and inverse images functors. 

Our aim is to show that, more generally, for \emph{any fibration} (i.e. not necessarily a cartesian one, nor a stack) $\mbc$ which is $J^{\textup{can}}_{\ce}$-small (in the sense of Definition \ref{J-smallgen}), we have a bipullback diagram:

\[\begin{tikzcd}
	{\mathbf{Gir}(f^*\mbc)} & {\mathbf{Gir}(\mbc)} \\
	\cf & \ce
	\arrow["{f^{\mbc}}", from=1-1, to=1-2]
	\arrow["{C_{p'}}"', from=1-1, to=2-1]
	\arrow["\lrcorner"{anchor=center, pos=0.125}, draw=none, from=1-1, to=2-2]
	\arrow["{C_p}", from=1-2, to=2-2]
	\arrow["f"', from=2-1, to=2-2]
\end{tikzcd}\]

A preliminary observation is that, denoting by $\zeta_{\mathbb C}$ the unit of the stackification-inclusion adjunction, the following equivalence holds (Proposition 6.3.8 of \cite{locfib}):

\[\begin{tikzcd}
	{\mathbf{Gir}(\cg(\mbc))} && {\mathbf{Gir}(\cg(s_J(\mbc)))} \\
	& {\widehat{\cc}_J}
	\arrow["{C_p}"', from=1-1, to=2-2]
	\arrow["{\Sh({\zeta_{\mbc}})}"', from=1-3, to=1-1]
	\arrow["\simeq"{description}, shift left=3, draw=none, from=1-3, to=1-1]
	\arrow["{C_{s_J(p)}}", from=1-3, to=2-2]
\end{tikzcd}\]

Let $F: (\cc,J) \to (\cc',J')$ be a morphism of sites, and let $p : \mathcal{G}(\mathbb C) \to \cc$ be a fibration. The inverse image of the stack $s_J(\mbc)$ along $F$ is given by $s_{J'} \circ \mathrm{Lan}_F (s_J \mbc)$. Since stackification with respect to $J$ or $J'$ does not affect the associated Giraud toposes (which are the objects of interest here), we may equivalently work with $\mathrm{Lan}_F(\mbc)$ or with $s_{J'} \mathrm{Lan}_F s_J (\mbc)$ when considering the inverse image $\St(F)^*(\mbc)$. For brevity, we will often denote this operation by $f^*\mbc$, as it extends the inverse image of sheaves to stacks.

The next subsections are devoted to prove the result that there is a (commutative) square 

\[\begin{tikzcd}
	{(\mathcal{G}(f^*(\mathbb C)),Gir_{f^*\mbc})} & {(\mathcal{G}(\mbc),Gir_{\mbc})} \\
	\cf & \ce
	\arrow["{p'}"', from=1-1, to=2-1]
	\arrow["{F^{\mbc}}"', from=1-2, to=1-1]
	\arrow["p", from=1-2, to=2-2]
	\arrow["{f^*}", from=2-2, to=2-1]
\end{tikzcd}\]
inducing a pullback square of geometric morphisms

\[\begin{tikzcd}
	{\mathbf{Gir}(f^*\mbc)} & {\mathbf{Gir}(\mbc)} \\
	\cf & \ce
	\arrow["{\Sh(F^{\mbc})}", from=1-1, to=1-2]
	\arrow["{C_{p'}}"', from=1-1, to=2-1]
	\arrow["{C_p}", from=1-2, to=2-2]
	\arrow[""{name=0, anchor=center, inner sep=0}, "f"', from=2-1, to=2-2]
	\arrow["\lrcorner"{anchor=center, pos=0.125}, draw=none, from=1-1, to=0]
\end{tikzcd}\]

In other words, the pullback of the Giraud topos of a fibration along a geometric morphism is the Giraud topos of its inverse image by this geometric morphism. 

To obtain pullbacks of relative presheaf toposes in full generality, a first step is to consider morphisms of fibrations of the form $\mathbf{Fib}_{\ce}(\mathcal{G}(\mathbb C),({\cg} / g^*f^*))$, where $\mathbb C$ is a $\ce$-indexed category, $f : {{\cal F}} \to \ce$ and $g : \cg \to \cf$ two relative toposes. Indeed, such functors can be transposed along the adjunctions induced by the base-change morphism $(f^*  \dashv f_*)$. The question is: does this transposition operation restrict to morphisms of sites? As previously noted, this is easy to verify when $\mbc$ is cartesian, since morphisms of sites are then characterized by a fibrewise preservation of finite limits. However, morphisms of sites arising from a fibration $\cg(\mbc)$ that is not cartesian are \emph{not} described by a fibrewise condition, which makes their transposition more subtle.

A natural first step is to consider the correspondence between morphisms of fibrations and indexed weak geometric morphisms, as established by the so-called indexed weak Diaconescu’s theorem (Theorem 6.3.5 of \cite{locfib}). Let $f : {\cal F} \to {\cal E}$ be a geometric morphism between two base toposes, and let $g : {\cal G} \to {\cal F}$ be a topos over ${\cal F}$. Given an ${\cal E}$-indexed category $\mathbb C$, and denoting by $p'$ the projection functor $p' : \cg(f^*\mbc) \to \cf$, the indexed weak Diaconescu’s theorem yields the following equivalences:

\[\begin{tikzcd}
	{\mathbf{IndWeakGeom}_{\cf}(S_{f^*\mbc},S_g)} & {\mathbf{IndWeakGeom}_{\ce}(S_{\mbc},S_{fg})} \\
	{\mathbf{Fib}_{\cf}([p'],[\pi_g])} & {\mathbf{Fib}_{\ce}([p],[\pi_{fg}])}
	\arrow["\simeq"{marking, allow upside down}, draw=none, from=1-1, to=1-2]
	\arrow["\simeq"{marking, allow upside down}, draw=none, from=1-1, to=2-1]
	\arrow["\simeq"{marking, allow upside down}, draw=none, from=1-2, to=2-2]
	\arrow["\simeq"{marking, allow upside down}, draw=none, from=2-1, to=2-2]
\end{tikzcd}\]

Indeed, the vertical equivalences follow from indexed weak Diaconescu's theorem, and the one at the bottom comes from to the adjointness of inverse and direct images along $f^*$. 

Hence, for the categories of canonical stacks over a base topos and indexed weak geometric morphisms between them, the operation of inverse image of an indexed presheaf topos is given by the presheaf topos on the inverse image, as described:
\begin{equation} \label{eq:weak_adjunction}
\begin{adjustbox}{scale=0.9,center}
\(
\mathbf{WeakIndGeom}_{\cf}(f^*_{\mathbf{Weak}}(S_{{\mbc}}),S_{g})
\simeq 
\mathbf{WeakIndGeom}_{\ce}(S_{{\mbc}},f_*^{\mathbf{Weak}}(S_{g}))
\)
\end{adjustbox}
\end{equation}
Here, $  f^*_{\mathbf{Weak}} \dashv f_*^{\mathbf{Weak}}$ are respectively the functors sending the canonical stack of a relative presheaf topos $S_{{\mbc}}$ to the canonical stack of the relative presheaf topos on the inverse image $S_{{f^*\mbc}}$, and a canonical stack $S_g$ to its direct image $S_{fg}$. This adjunction parallels the first result of Section 3 in \cite{pitts}, but is formulated in terms of canonical stacks rather than cocomplete indexed categories, and involves fibrations instead of internal categories.

In order to obtain pullbacks of relative presheaf toposes, and in light of the equivalence between relative geometric morphisms and indexed geometric morphisms discussed in subsection 6.1 of \cite{locfib}, one strategy is to show that equivalence \ref{eq:weak_adjunction} restricts to those indexed weak geometric morphisms that are in particular indexed geometric morphisms.

From Proposition 3.4 (iii) \cite{bartolicaramello}, it is known that a \emph{fibrewise preservation of finite limits} characterization for morphisms of sites holds - albeit at a higher level of abstraction, namely, the canonical stack. Indeed, the theorem asserts that a morphism of fibrations is a morphism of sites if and only if its $\eta$-extension is preserves finite limits fibrewise. As mentioned in the previous subsection, transpositions along direct and inverse images preserve the fibrewise finite limits preservation; the main obstruction lies in the subtle interaction between the inverse image operation and the canonical stack construction, on which the notion of $\eta$-extension is based. A priori, the operations of taking the canonical stack and taking the inverse image do not commute in a way that would ensure direct compatibility with the notion of $\eta$-extension. However, the introduction of an appropriate topology will eventually turn the inverse image of the canonical stack and the canonical stack of the inverse image into Morita-equivalent relative sites. This will in turn yield that the adjunction induced by inverse and direct images along $f$ transposes morphisms of sites into morphisms of sites:

\[
\mathbf{FibSites}_{\cf}([p_{f^*\mathbb{C}}],[\pi_g]) \simeq \mathbf{FibSites}_{\ce}([p_{\mathbb{C}}],[\pi_{fg}]).
\]
The full situation is pictured in the following diagram:
\vspace{0.5cm}

\begin{adjustbox}{scale=0.5}
$\begin{tikzcd}[column sep=tiny]
	{{\mathbf{IndWeakGeom}_{{\cal F}}(S_g,S_{C_{p'}})}} &&& {\mathbf{IndWeakGeom}_{{\cal E}}(S_{fg},S_{C_{p}})} \\
	& {{\mathbf{IndGeom}_{{\cal F}}(S_g,S_{C_{p'}})}} & {{\mathbf{IndGeom}_{{\cal E}}(S_{fg},S_{C_{p}})}} \\
	& {\mathbf{FibSites}_{{\cal F}}((\mathcal{G}(f^*\mathbb C),Gir_{f^*\mathbb C}),(({\cg}/g^*),J_g))} & {\mathbf{FibSites}_{{\cal E}}((\mathcal{G}(\mathbb C),Gir_{\mathbb C}),(({\cg}/g^*f^*),J_{fg}))} \\
	{\mathbf{Fib}_{{\cal F}}({\cg}(f^*\mathbb C),({\cg}/g^*))} &&& {\mathbf{Fib}_{{\cal E}}({\cg}(\mathbb C),({\cg}/g^*f^*))}
	\arrow["{Ind.Weak.Diac.}"{description}, shift right=14, draw=none, from=1-1, to=4-1]
	\arrow["\simeq"{description}, from=1-4, to=1-1]
	\arrow[hook', from=2-2, to=1-1]
	\arrow["{Ind.Diac.}"', shift right=3, draw=none, from=2-2, to=3-2]
	\arrow[hook, from=2-3, to=1-4]
	\arrow["{?}"', dashed, no head, from=2-3, to=2-2]
	\arrow["{Ind.Diac.}", shift left=3, draw=none, from=2-3, to=3-3]
	\arrow["\simeq"{description}, no head, from=3-2, to=2-2]
	\arrow[hook, from=3-2, to=4-1]
	\arrow["\simeq"{description}, no head, from=3-3, to=2-3]
	\arrow["{?}"', dashed, no head, from=3-3, to=3-2]
	\arrow[hook', from=3-3, to=4-4]
	\arrow["\simeq"{description}, no head, from=4-1, to=1-1]
	\arrow["\simeq"{description}, no head, from=4-4, to=1-4]
	\arrow["{Ind.Weak.Diac.}"{description}, shift right=14, draw=none, from=4-4, to=1-4]
	\arrow["\simeq"{description}, from=4-4, to=4-1]
	\arrow["transposition"{description}, shift left=5, draw=none, from=4-4, to=4-1]
\end{tikzcd}$
\end{adjustbox}
\vspace{0.3cm}\\
Here, the vertical equivalences are given by relative Diaconescu's theorem and indexed weak Diaconescu's theorem, the lower horizontal equivalence is given by transposition. Building on the preceding discussion, the remainder of the paper is devoted to establishing that the equivalence given by transposition restricts to morphisms of sites, that is, to proving the existence of the dashed arrows labelled \emph{?} in the above diagram.

\subsection{The structure morphism}\label{subsec3}

In this subsection, the structure morphism $f^{\mbc}$ is investigated, relating the pullback of a Giraud topos (i.e., a relative presheaf topos) to the original one, as depicted in the following diagram:

\[\begin{tikzcd}
	{\mathbf{Gir}(f^*\mathbb C)} & {\mathbf{Gir}(\mathbb C)} \\
	{{\cal F}} & {{\cal E}}
	\arrow["{f^{\mbc}}", from=1-1, to=1-2]
	\arrow["{C_{p_{f^*\mathbb C}}}"', from=1-1, to=2-1]
	\arrow["{C_{p_{\mathbb C}}}", from=1-2, to=2-2]
	\arrow["f"', from=2-1, to=2-2]
\end{tikzcd}\]

The morphism is shown to be induced by the functor $F^{\mathbb C} = q^{f^*}_{\mbc}\varsigma_{\mathbb C}^{f}$, as depicted in the following square:
\[\begin{tikzcd}
	{\mathcal{G}(f^*\mathbb C)} & {{\cg}((f^*\mathbb C)f^*)} & {\mathcal{G}(\mathbb C)} \\
	{{{\cal F}}} & {{{\cal E}}}
	\arrow["{p'}"', from=1-1, to=2-1]
	\arrow["{q^{f^*}_{\mbc}}"', from=1-2, to=1-1]
	\arrow["\lrcorner"{anchor=center, pos=0.125, rotate=-90}, draw=none, from=1-2, to=2-1]
	\arrow[from=1-2, to=2-2]
	\arrow["{\varsigma_{\mathbb C}^{f}}"', from=1-3, to=1-2]
	\arrow["p", from=1-3, to=2-2]
	\arrow["{f^*}"', from=2-2, to=2-1]
\end{tikzcd}\]
\noindent where $\varsigma_{\mathbb C}^{f}$ is the unit of the adjunction $f^* \dashv f_*$, and $q^{f^*}_{\mbc}$ is the projection from the pullback in $\textup{\bf Cat}$ (see Definition \ref{defimdirecte}).

Moreover, the Giraud topos construction is shown to be invariant under dense morphisms of sites.

In the case of cartesian fibrations, one can readily verify that both functors $q^{f^*}_{\mbc}$ and $\varsigma_{\mathbb C}^{f}$ preserve finite limits: for the pullback projection $q^{f^*}_{\mbc}$ it is the Proposition \ref{imdirectecartreflectlimfin}, while for the unit $\varsigma_{\mathbb C}^{f}$ it follows from the fact that it is the transpose of the identity morphism, which is a cartesian morphism of cartesian fibrations. Moreover, since $\varsigma_{\mathbb C}^{f}$ and $q^{f^*}_{\mbc}$ are morphisms of fibrations, their composition $F^{\mbc}$ preserves covers for Giraud topologies, and is thus a morphism of sites. In this cartesian setting, this morphism of sites is precisely the one inducing the projection geometric morphism denoted by $f_{\mbc}$. Crucially, this proof does not require any explicit description of the inverse image fibration $f^*\mbc$: it solely relies on the fact that finite limits are computed fibrewise, and that morphisms of cartesian fibrations preserve them fibrewise. In the general case of a non-cartesian fibration, this reasoning no longer applies. One possible approach is to compute the inverse image $f^*\mbc$ explicitly; however, as mentioned after Proposition \ref{inverseimagepropdef} this leads to a rather heavy computation. Combined with the general concrete definition of morphisms of sites (involving locally filtering conditions) this would make the process particularly cumbersome. The following result suggests that working with comorphisms of sites rather than with morphisms of sites (in particular $f^*$) offers a more manageable framework, and will naturally lead us to exploit the abstract duality of \cite{denseness} between morphisms and comorphisms of sites.

\begin{prop}\label{pullbackcomorphalongfib}
Let $G : ({\cd},K) \to ({\cc},J)$ be a comorphism of sites and $\mathbb C$ a $\cc$-indexed category. In the following pseudopullback square
\[\begin{tikzcd}
	{(\mathcal{G}(\mathbb CG),Gir_{\mathbb CG})} & {(\mathcal{G}(\mathbb C),Gir_{\mathbb C})} \\
	{({\cd},K)} & {({\cc},J)}
	\arrow["{q^G_{\mbc}}", from=1-1, to=1-2]
	\arrow["{p_{\mathbb CG}}"', from=1-1, to=2-1]
	\arrow["\lrcorner"{anchor=center, pos=0.125}, draw=none, from=1-1, to=2-2]
	\arrow["{p_{\mathbb C}}", from=1-2, to=2-2]
	\arrow["G"', from=2-1, to=2-2]
\end{tikzcd}\]
the functor $q^G_{\mbc}$ is a comorphism of sites.
\end{prop}

\begin{proof}
If we have some covering for $Gir_{\mathbb C}$ of an object $(x,G(d))$ as follows: $((1,v_i) : (\mathbb C (v_i)(x),c_i) \to (x,G(d)))_i$ we can lift the covering $(v_i)_i$ through the comorphism $G$: we have a covering $(u_j)_j$ for $K$ such that for each $j$ there exist a $i$ and an arrow $a_j$ such that $G(u_j) = a_jv_i$. This covering can be lifted into $\mathcal{G}(\mathbb CG)$: $((1,u_j) : (\mathbb CG(u_j)(x),dom(u_j)) \to (x,d))_j $. This covering is sent by $G_{\mathbb C}$ into the sieve generated by the $(1,v_i)_i$, yielding a lifting for the initial covering.
\end{proof}

The key point is that up to stackification, which does not affect the constructions at the topos level, since a fibration and its stackification yield the same relative presheaf topos, $\mathbb CG$ is the inverse image of $\mbc$ along the geometric morphism induced by $G$ as a comorphism. Since inverse images of geometric morphisms arising from comorphisms correspond to pseudopullbacks at the fibration level (or to precomposition at the indexed level), they are easier to compute and manipulate. In contrast, deducing that the functor linking a fibration to its inverse image along a geometric morphism induced by a morphism of sites (such as an inverse image functor) itself induces a geometric morphism at the topos level would have been significantly more involved.

\begin{remarks}
\begin{enumerate}[(a)]
    \item Since all four functors in the previous proposition are comorphisms of sites and make the square commute at the site level, they induce, by functoriality, a commutative square at the topos level.
    \item If the comorphism of sites $p$ is a left adjoint, the inverse image along its right adjoint $F$ is computed as the pullback along $p$ (see Proposition \ref{adjointsinversedirectimage}). This yields a \emph{pair} of adjoint functors comparing the fibration and its inverse image:

\[\begin{tikzcd}
	{\mathcal{G}(Lan_F \mathbb C)} & {\mathcal{G}(\mathbb Cp)} & {\mathcal{G}(\mathbb C)} \\
	& {({\cd},K)} & {({\cc},J)}
	\arrow["\simeq"{description}, draw=none, from=1-1, to=1-2]
	\arrow[""{name=0, anchor=center, inner sep=0}, "{q^p_{\mathbb C}}", shift left=3, from=1-2, to=1-3]
	\arrow[from=1-2, to=2-2]
	\arrow[""{name=1, anchor=center, inner sep=0}, "{(Lan_F)^{\mathbb C}}", shift left=2, from=1-3, to=1-2]
	\arrow[from=1-3, to=2-3]
	\arrow[""{name=2, anchor=center, inner sep=0}, "p", shift left=2, from=2-2, to=2-3]
	\arrow[""{name=3, anchor=center, inner sep=0}, "F", shift left=2, from=2-3, to=2-2]
	\arrow["\dashv"{anchor=center, rotate=-90}, draw=none, from=0, to=1]
	\arrow["\dashv"{anchor=center, rotate=-90}, draw=none, from=2, to=3]
\end{tikzcd}\]

    acting as follows: ${q}^p_{\mathbb C}$ is the usual projection of the pullback, sending an object $(x,d)$ to $(x,p(d))$, and $(Lan_F)^{\mathbb C}$ sends an object $(x,c)$ to the object $(\mathbb C(\epsilon_c)(x),F(E))$. The adjunction between them follows directly from the adjunction between $F$ and $p$.
.
    For example, for the case of an essential geometric morphism $f$, that is with an exceptional left adjoint $f_! \dashv  f^*$, there is an expression of $f^*\mathbb C$ as the following pullback, which now comes equipped with two adjoint functors
\[\begin{tikzcd}
	{\mathcal{G}(\mathbb C)} & {\mathcal{G}(\mathbb Cf_!)} \\
	{{{\cal E}}} & {{{\cal F}}}
	\arrow[""{name=0, anchor=center, inner sep=0}, "{f^*{\mathbb C}}", shift left=2, from=1-1, to=1-2]
	\arrow[from=1-1, to=2-1]
	\arrow[""{name=1, anchor=center, inner sep=0}, "{{f_!}_{\mathbb C}}", shift left=2, from=1-2, to=1-1]
	\arrow["\lrcorner"{anchor=center, pos=0.125, rotate=-90}, draw=none, from=1-2, to=2-1]
	\arrow[from=1-2, to=2-2]
	\arrow["{f_!}", from=2-2, to=2-1]
	\arrow["\dashv"{anchor=center, rotate=90}, draw=none, from=1, to=0]
\end{tikzcd}\]
    \end{enumerate}
\end{remarks}

Within the framework of the abstract duality developed in \cite{denseness}, the following notations will be adopted:
    \[\begin{tikzcd}
	{\mathcal{G}({\mathbb C}\pi_f)} & {\mathcal{G}(Lan_{i_f}{\mathbb C})} && {\mathcal{G}(\mathbb C)} \\
	& {({{\cal F}}/ f^*)} && {{{\cal E}}}
	\arrow["\simeq"{description}, draw=none, from=1-1, to=1-2]
	\arrow[""{name=0, anchor=center, inner sep=0}, "{\pi_f^{\mathbb C}}"', shift right, from=1-2, to=1-4]
	\arrow[from=1-2, to=2-2]
	\arrow[""{name=1, anchor=center, inner sep=0}, "{i_f^{\mathbb C}}"', shift right=2, from=1-4, to=1-2]
	\arrow[from=1-4, to=2-4]
	\arrow[""{name=2, anchor=center, inner sep=0}, "{\pi_f}"', from=2-2, to=2-4]
	\arrow[""{name=3, anchor=center, inner sep=0}, "{i_f}"', shift right=3, from=2-4, to=2-2]
	\arrow["\dashv"{anchor=center, rotate=90}, draw=none, from=0, to=1]
	\arrow["\dashv"{anchor=center, rotate=90}, draw=none, from=2, to=3]
    \end{tikzcd}\]    

According to the preceding proposition, when the inverse image of a fibration is computed via a comorphism of sites, one obtains a geometric morphism between the Giraud topos of the fibration and that of its inverse image. Leveraging the abstract duality, the goal is to deduce an analogous result in the case of morphisms of sites. The following diagram relates the two situations:

\[\begin{tikzcd}
	& {\mathcal{G}(Lan_{i_f}{\mathbb C})} \\
	& {({{\cal F}}\downarrow f^*)} \\
	{\mathcal{G}(f^*\mathbb C)} && {\mathcal{G}(\mathbb C)} \\
	{{{\cal F}}} && {{{\cal E}}}
	\arrow[from=1-2, to=2-2]
	\arrow["{\pi_{{\cal F}}^{\mathbb C}}"', from=1-2, to=3-1]
	\arrow[""{name=0, anchor=center, inner sep=0}, "{\pi_f^{\mathbb C}}"{description, pos=0.3}, from=1-2, to=3-3]
	\arrow["{\pi_{{\cal F}}}"{description, pos=0.3}, from=2-2, to=4-1]
	\arrow[""{name=1, anchor=center, inner sep=0}, "{\pi_f}"', from=2-2, to=4-3]
	\arrow[from=3-1, to=4-1]
	\arrow[""{name=2, anchor=center, inner sep=0}, "{i_f^{\mathbb C}}"{description, pos=0.3}, shift right=4, from=3-3, to=1-2]
	\arrow["{F^{\mbc}}"', from=3-3, to=3-1]
	\arrow[from=3-3, to=4-3]
	\arrow[""{name=3, anchor=center, inner sep=0}, "{i_f}"', shift right=3, from=4-3, to=2-2]
	\arrow["{f^*}"', from=4-3, to=4-1]
	\arrow["\dashv"{anchor=center, rotate=34}, draw=none, from=0, to=2]
	\arrow["\dashv"{anchor=center, rotate=34}, draw=none, from=1, to=3]
\end{tikzcd}\]

It was previously established (see Proposition \ref{pullbackcomorphalongfib}) that the functor relating a fibration to its inverse image, when computed via a comorphism of sites, is itself a comorphism of sites, and hence induces a geometric morphism between the associated relative presheaf toposes. The abstract duality of \cite{denseness} (as recalled in Theorem \ref{canonicalrelativesiteMoritaequi}) provides a Morita equivalence between the morphism of sites $f^*$ and the comorphism $\pi_f$. It is thus natural to expect that the property established for comorphisms - here, for $\pi_f$ - also holds for the original morphism $f^*$. To address this question, the analysis now turns to the behavior of the functor $\pi_{\cf}^{\mbc}$, in connection with the invariance of relative presheaf toposes under dense morphisms of sites.

\begin{prop}\label{invariancegiraud}
Let $({\cc},J)$ be a site and $\mathbb C$ a $\cc$-indexed category. They provide the following square of functors:

\[\begin{tikzcd}
	{\mathcal{G}(\mathbb C)} & {\mathcal{G}(Lan_{l_J}\mathbb C)} \\
	{({\cc},J)} & {\widehat{\cc}_J}
	\arrow["{l^{\mathbb C}}", from=1-1, to=1-2]
	\arrow[from=1-1, to=2-1]
	\arrow[from=1-2, to=2-2]
	\arrow["l"', from=2-1, to=2-2]
\end{tikzcd}\]

This square induces, at the topos-level, the following equivalences:

\[\begin{tikzcd}
	{\mathbf{Gir}(\mathbb C)} & {\mathbf{Gir}(Lan_{l_J}\mathbb C)} \\
	{\widehat{\cc}_J} & {\widehat{\widehat{\cc}_J}_{J^{\textup{can}}_{\widehat{\cc}_J}}}
	\arrow["{C_{l^{\mathbb C}}}", shift left=4, from=1-1, to=1-2]
	\arrow["\simeq"{description}, shift left=2, draw=none, from=1-1, to=1-2]
	\arrow[from=1-1, to=2-1]
	\arrow["{\Sh(l^{\mathbb C})}", shift left, from=1-2, to=1-1]
	\arrow[from=1-2, to=2-2]
	\arrow["{C_{l}}", shift left, from=2-1, to=2-2]
	\arrow["\simeq"{description}, shift right, draw=none, from=2-1, to=2-2]
	\arrow["{\Sh(l)}", shift left=3, from=2-2, to=2-1]
\end{tikzcd}\]
\end{prop}

\begin{proof}
The fundamental adjunction of Corollary 5.3.7. \cite{CaramelloZanfa}, for a site $({\cc},J)$ and the canonical site of its associated topos $(\widehat{\cc}_J,J^{\textup{can}}_{\widehat{\cc}_J})$, ensures the following square of adjunctions:
\[\begin{tikzcd}
	{\mathbf{Ind}_{\cc}} && {\mathbf{Top}/{\widehat{\cc}_J}} \\
	\\
	{\mathbf{Ind}_{\widehat{\cc}_J}} && {\mathbf{Top}/\widehat{{\widehat{\cc}_J}}_{J_{\widehat{\cc}_J}^{\textup{can}}}}
	\arrow[""{name=0, anchor=center, inner sep=0}, "{\Lambda_{{\cc}}}", shift left=2, from=1-1, to=1-3]
	\arrow[""{name=1, anchor=center, inner sep=0}, "{Lan_{l^{\textup{op}}}}"', shift right=3, from=1-1, to=3-1]
	\arrow[""{name=2, anchor=center, inner sep=0}, "{\Gamma_{\cc}}", shift left, from=1-3, to=1-1]
	\arrow[""{name=3, anchor=center, inner sep=0}, "{(C_l\circ-)}"', shift right=3, from=1-3, to=3-3]
	\arrow[""{name=4, anchor=center, inner sep=0}, "{l_*}"', from=3-1, to=1-1]
	\arrow[""{name=5, anchor=center, inner sep=0}, "{\Lambda_{{\widehat{\cc}_J}}}", shift left, from=3-1, to=3-3]
	\arrow[""{name=6, anchor=center, inner sep=0}, "{(\Sh(l)\circ-)}"', from=3-3, to=1-3]
	\arrow[""{name=7, anchor=center, inner sep=0}, "{\Gamma_{{\widehat{\cc}_J}}}", shift left=2, from=3-3, to=3-1]
	\arrow["\dashv"{anchor=center, rotate=-90}, draw=none, from=0, to=2]
	\arrow["\dashv"{anchor=center}, draw=none, from=1, to=4]
	\arrow["\dashv"{anchor=center}, draw=none, from=3, to=6]
	\arrow["\dashv"{anchor=center, rotate=-90}, draw=none, from=5, to=7]
\end{tikzcd}\]

Here, the adjoints of the right face of the square are equivalences, as $l$ is a dense morphism and comorphism of sites. Now, it obviously commutes for right adjoints, that is $l_*\Gamma_{\widehat{\cc}_J} \simeq \Gamma_{\cc}(\Sh(l)\circ -)$, so that it also does for the left adjoints which associate its Giraud topos to an indexed category; hence the Giraud topos of an indexed category $\mbc$ is equivalent to the one of its extension $Lan_{l}(\mbc)$.
\end{proof}

In order to obtain a similar result for arbitrary morphisms of sites, the following funtoriality result for the functor relating a fibration to its inverse image will be useful:

\begin{prop}\label{functorialityofstructuralfunctorinverseimage}
Let $F : {\cc} \to {\cc'}$ and $F': {\cc'} \to {\cc''}$ be two functors, together with $\mathbb C$ a $\cc$-indexed category. The following isomorphism holds: $(Lan_{F'F})^{\mathbb C} \simeq (Lan_F)_{\mathbb C}(Lan_{F'})^{Lan_F\mathbb C}$, as pictured in the diagram

\[\begin{tikzcd}
	{\mathcal{G}(\mathbb C)} & {\mathcal{G}(F\mathbb C)} & {\mathcal{G}(F'F\mathbb C)} \\
	{{\cal C}} & {{\cal C'}} & {{\cal C''}}
	\arrow["{(Lan_F)^{\mathbb C}}"', shift right, from=1-1, to=1-2]
	\arrow["{(Lan_{F'F})^{\mathbb C}}", bend left = 18, from=1-1, to=1-3]
	\arrow[from=1-1, to=2-1]
	\arrow["{(Lan_{F'})^{Lan_F\mathbb C}}"', shift right, from=1-2, to=1-3]
	\arrow[from=1-2, to=2-2]
	\arrow[from=1-3, to=2-3]
	\arrow["F"', from=2-1, to=2-2]
	\arrow["{F'}"', from=2-2, to=2-3]
\end{tikzcd}\]
\end{prop}

\begin{proof}
First, the result is established for $\mathbb C$ representable; a colimit argument then extends it to arbitrary fibrations.

So let $c$ be an object of ${\cc}$ and $\mathbb C$ the associated representable. The associated diagram is

\[\begin{tikzcd}
	{{\cc}/c} & {{\cc'}/F(c)} & {{\cc''}/F'F(c)} \\
	{{\cc}} & {{\cc'}} & {{\cc''}}
	\arrow["{F^c}", from=1-1, to=1-2]
	\arrow["{(F'F)^{c}}"{description}, bend left = 25, from=1-1, to=1-3]
	\arrow[from=1-1, to=2-1]
	\arrow["{F'^{F(c)}}", from=1-2, to=1-3]
	\arrow[from=1-2, to=2-2]
	\arrow[from=1-3, to=2-3]
	\arrow["F"', from=2-1, to=2-2]
	\arrow["{F'}"', from=2-2, to=2-3]
\end{tikzcd}\]
where the upper functor send $[u : \overline{c} \to c]$ to $[F'F(u)]$, and $F^c : [u] \mapsto [F(u)]$, and $(F')^{F(c)} : [u'  : c' \to F(c)] \mapsto [F'(u')]$. Hence, the proposition is verified for these functors.

Since $\mathcal{G}(\mathbb C) \simeq \colim_{lax}(\mathbb C)$ (see Proposition 2.9.5. \cite{CaramelloZanfa}), and the inverse image functor commutes with lax colimits - as a left biadjoint -, it follows that the commutation of the following squares results from the very definition of the functors $(\mathrm{Lan}F)^{\mathbb C}$, together with the naturality of the unit of the adjunction $\mathrm{Lan}_F \dashv (-\circ F)$.

\[\begin{tikzcd}
	{\mathcal{G}(\mathbb C)} & {\mathcal{G}(Lan_F(\mathbb C))} \\
	{{\cc}/c} & {{\cc'}/F(c)}
	\arrow["{(Lan_F)^{\mathbb C}}", from=1-1, to=1-2]
	\arrow[from=2-1, to=1-1]
	\arrow["{F^c}"', from=2-1, to=2-2]
	\arrow[from=2-2, to=1-2]
\end{tikzcd}\]
In this square, the vertical functors are determined by the fibered Yoneda lemma (see Proposition 2.3.1 \cite{CaramelloZanfa}). Moreover, every natural transformation involved in the lax colimit presentation of $\mathcal{G}(\mathbb C)$ gives rise to a commutative square, compatible with the previous ones.

Finally, since the identity $(F'F)^c \simeq F^{\mathbb C}F'^{F(c)}$ holds for representable presheaves, and every fibration $\mathcal{G}(\mathbb C)$ is the lax colimit of representable presheaves, with the associated lax cocones being compatible with the functors $F^c$, this yields the identity $(F'F)^{\mathbb C} \simeq F^{\mathbb C}F'^{F\mathbb C}$.
\end{proof}

\begin{cor}
Let $i : ({\cc},J) \to ({\cc}',J')$ be a morphism of sites inducing an equivalence at the level of toposes, and $\mathbb C$ a $\cc$-indexed category. Then, there is the following commutative square at the level of toposes:

\[\begin{tikzcd}
	{\mathbf{Gir}(Lan_i\mathbb C)} & {\mathbf{Gir}(\mathbb C)} \\
	{\widehat{\cc'}_{J'}} & {\widehat{\cc}_J}
	\arrow["{\Sh({i^{\mathbb C}})}", from=1-1, to=1-2]
	\arrow["\simeq"{marking, allow upside down}, shift right=2, draw=none, from=1-1, to=1-2]
	\arrow[from=1-1, to=2-1]
	\arrow[from=1-2, to=2-2]
	\arrow["{\Sh({i})}", shift right, from=2-1, to=2-2]
	\arrow["\simeq"{description}, shift right=3, draw=none, from=2-1, to=2-2]
\end{tikzcd}\]
\end{cor}

\begin{proof}

Because of the previous proposition, the isomorphism $(\Sh(i)^*)^{\widetilde{\mbc}}l^{\mbc} \simeq l'^{Lan_i\mbc}i^{\mbc}$ holds, so that the following diagram is commutative:
\[\begin{tikzcd}
	{\mathcal{G}(\widetilde{\mathbb C})} &&& {\mathcal{G}(\widetilde{Lan_i\mathbb C})} \\
	& {\mathcal{G}(\mathbb C)} & {\mathcal{G}(Lan_i\mathbb C)} \\
	& {({\cc},J)} & {({\cc}',J')} \\
	{\widehat{\cc}_J} &&& {\widehat{\cc'}_{J'}}
	\arrow["{(\Sh(i)^*)^{\widetilde{\mbc}}}", from=1-1, to=1-4]
	\arrow[from=1-1, to=4-1]
	\arrow[from=1-4, to=4-4]
	\arrow["{l^{\mathbb C}}", from=2-2, to=1-1]
	\arrow["{i^{\mathbb C}}", from=2-2, to=2-3]
	\arrow[from=2-2, to=3-2]
	\arrow["{l'^{Lan_i\mathbb C}}"', from=2-3, to=1-4]
	\arrow[from=2-3, to=3-3]
	\arrow["i"', from=3-2, to=3-3]
	\arrow["l"', from=3-2, to=4-1]
	\arrow["{l'}", from=3-3, to=4-4]
	\arrow["{\Sh(i)^*}"', from=4-1, to=4-4]
	\arrow["\simeq"{description}, shift left=3, draw=none, from=4-1, to=4-4]
\end{tikzcd}\]

Note that $\St(\Sh(i)^*)\St(l)^* \simeq \St(l')^*\St(i)^*$ holds because of the functoriality of $\St(-)^*$: this justifies that $(\Sh(i)^*)^{\widetilde{\mbc}} : \cg(\widetilde{\mbc}) \to \cg(\widetilde{Lan_i(\mbc)}$ is an equivalence (up to stackification, which does not change anything when passing to relative presheaf toposes). Hence, this diagram induces, at the topos-level, the following commutative diagram:

\[\begin{tikzcd}
	{\mathbf{Gir}(\widetilde{\mathbb C})} &&& {\mathbf{Gir}(\widetilde{Lan_i\mathbb C})} \\
	& {\mathbf{Gir}(\mathbb C)} & {\mathbf{Gir}(Lan_i\mathbb C)} \\
	& {\widehat{\cc}_J} & {\widehat{\cc'}_{J'}} \\
	{\widehat{\cc}_J} &&& {\widehat{\cc'}_{J'}}
	\arrow["{\Sh(l^{\mathbb C})}"', from=1-1, to=2-2]
	\arrow["\simeq", draw=none, from=1-1, to=2-2]
	\arrow[from=1-1, to=4-1]
	\arrow["{\Sh((\Sh(i)^*)^{\widetilde{\mathbb C}})}"', from=1-4, to=1-1]
	\arrow["\simeq", draw=none, from=1-4, to=1-1]
	\arrow["{\Sh(l'^{Lan_i\mathbb C})}", from=1-4, to=2-3]
	\arrow["\simeq"', draw=none, from=1-4, to=2-3]
	\arrow[from=1-4, to=4-4]
	\arrow["{\Sh(i^{\mathbb C})}", from=2-2, to=2-3]
	\arrow[from=2-2, to=3-2]
	\arrow[from=2-3, to=3-3]
	\arrow["{\Sh(i)}", from=3-3, to=3-2]
	\arrow["\simeq"', draw=none, from=3-3, to=3-2]
	\arrow["{\Sh(l)}", from=4-1, to=3-2]
	\arrow["\simeq"'{pos=0.7}, shift left, draw=none, from=4-1, to=3-2]
	\arrow["\simeq"{description}, shift left=3, draw=none, from=4-1, to=4-4]
	\arrow["{\Sh(l')}"', from=4-4, to=3-3]
	\arrow["\simeq"{pos=0.6}, shift right=2, draw=none, from=4-4, to=3-3]
	\arrow["{\Sh(i)}", from=4-4, to=4-1]
\end{tikzcd}\]

Indeed, the fact that the left and right squares are commutative is given by the invariance of the Giraud topos \ref{invariancegiraud}. The upper face square is commutative because of the functoriality of morphisms of sites, so that $\Sh(i_{\mathbb c})$ is forced to be an equivalence.
\end{proof}

Alternatively, if one does not wish to rely on the abstract result of the fundamental adjunction, a more concrete proof suffices in this case:

\begin{prop}\label{pullbackdensemorphsites}
Let $F : ({\cc'},J') \to ({\cc},J)$ be a dense morphism of sites (Definition 5.1. \cite{denseness}), and $\mathbb C$ a $\cc$-indexed category. Then, the functor $q_{\mathbb C}^F : (\mathcal{G}(\mathbb C F),Gir_{\mbc F}) \to (\mathcal{G}(\mathbb C),Gir_{\mbc})$ is a dense morphism of sites. It induces the following commutative square of toposes:

\[\begin{tikzcd}
	{\mathbf{Gir}(\mathbb CF)} & {\mathbf{Gir}(\mathbb C)} \\
	{\widehat{\cc'}_{J'}} & {\widehat{\cc}_J}
	\arrow["{C_{p'}}"', from=1-1, to=2-1]
	\arrow["{\Sh(q_{\mathbb C}^F)}"', from=1-2, to=1-1]
	\arrow["\simeq"{description}, shift left=3, draw=none, from=1-2, to=1-1]
	\arrow["{C_p}", from=1-2, to=2-2]
	\arrow["{\Sh(F)}", from=2-2, to=2-1]
	\arrow["\simeq"', draw=none, from=2-2, to=2-1]
\end{tikzcd}\]
\end{prop}

\begin{proof}
The definition of dense morphisms of sites used here is the one given in Definition 5.1 of \cite{denseness}, and all its conditions must be verified.

First, the reflection of covering is immediate since $F$ does reflect them, and the topologies are the Giraud ones.

Now let $(x,c)$ be an object of $\mathcal{G}(\mathbb C)$. In the basis, the object $c$ can be covered with objects coming from $\cc'$: $(v_i : F(c'_i) \to c)_i$ by $F$ being dense, yielding a covering $(q_{\mathbb C}^F(\mathbb C(v_i)(x),c_i') \to (x,c))_i$ for $Gir_{\mbc}$: this ensures that every object of $\mathcal{G}(\mathbb C)$ ca be covered by objects coming from $\mathcal{G}(\mathbb CF)$, together with the first condition for $q_{\mathbb C}^F$ to be a morphism of sites.

For two generalized elements as follows: $(u_1,g_1) : (x,c) \to F_{\mathbb C}(x_1,c'_1)$ and $(u_2,g_2) : (x,c) \to q_{\mathbb C}^F(x_2,c'_2)$, the denseness of $F$ provides a $J$-covering $(v_i : c_i \to c)_i$ such that $g_1v_i = F(g_1^i)$ and $g_2v_i = F(g_2^i)$. This allows to form the following product-shape diagram of generalized elements:

\[\begin{tikzcd}
	& {q_{\mathbb C}^F(x_1,c'_1)} \\
	{(x,c)} & {q_{\mathbb C}^F(x_2,c'_2)} \\
	{(\mathbb C(v_i)(x),F(c'_i))} & {q_{\mathbb C}^F(\mathbb C(v_i)(x),c'_i)}
	\arrow["{(u_1,g_1)}", from=2-1, to=1-2]
	\arrow["{(u_2,g_2)}", from=2-1, to=2-2]
	\arrow["{(1,v_i)}", from=3-1, to=2-1]
	\arrow[equals, from=3-1, to=3-2]
	\arrow["{q_{\mathbb C}^F(\mathbb C(v_i)(u_2),g_2^i)}", from=3-2, to=2-2]
\end{tikzcd}\]

The third condition for $q_{\mathbb C}^F$ to be a morphism of sites goes the same way.

For the $Gir_{\mbc F}$-fullness, let $(u,g): q_{\mathbb C}^F(x_1,c'_1) \to  q_{\mathbb C}^F(x_2,c'_2)$ be an arrow in $\mathcal{G}(\mathbb C)$. By the $J'$-fullness of $F$, there is a $J'$-covering $(v_i c'_i \to c'_1)$ and a family of arrows $(g'_i : c'_i \to c'_2)_i$ such that $gF(v_i) = F(g'_i)$. This leads to the identity $(u,g)q_{\mathbb C}^F(1,v_i)=q_{\mathbb C}^F(\mathbb C(v_i)(u),g'_i)$, which is the local fullness for $q_{\mathbb C}^F$. 

The square induced at the topos-level is commutative because of Proposition \ref{charactgeomtrelmorphfib} (iv): indeed, $q^F_{\mbc}$ is a morphism of sites which is moreover a morphism of fibrations with base change, so that it induces a relative geometric morphism.
\end{proof}

Everything is now in place to deduce:

\begin{prop}\label{F^Cmorphsites}
Let $f : {{\cal F}} \to {{\cal E}}$ be a relative topos, and $\mathbb C$ a ${\cal E}$-indexed category. The structural morphism pictured in the following diagram:

\[\begin{tikzcd}
	{\mathcal{G}(f^*\mathbb C)} & {\mathcal{G}(\mathbb C)} \\
	{{{\cal F}}} & {{{\cal E}}}
	\arrow["{p'}"', from=1-1, to=2-1]
	\arrow["{F^{\mathbb C}}"', from=1-2, to=1-1]
	\arrow["p", from=1-2, to=2-2]
	\arrow["{f^*}"', from=2-2, to=2-1]
\end{tikzcd}\]

\noindent is a morphism of sites, inducing a commutative square at the topos-level:

\[\begin{tikzcd}
	{\mathbf{Gir}(f^*\mbc)} & {\mathbf{Gir}(\mbc)} \\
	{{{\cal F}}} & {{{\cal E}}}
	\arrow["{f^{\mathbb C}}", from=1-1, to=1-2]
	\arrow["{C_{p'}}"', from=1-1, to=2-1]
	\arrow["{C_p}", from=1-2, to=2-2]
	\arrow["f"', from=2-1, to=2-2]
\end{tikzcd}\]

\end{prop}

\begin{proof}
Consider the following diagram given by the abstract duality of \cite{denseness}:

\[\begin{tikzcd}
	& {\mathcal{G}(Lan_{i_f}{\mathbb C})} \\
	& {({{\cal F}}/ f^*)} \\
	{\mathcal{G}(f^*\mathbb C)} && {\mathcal{G}(\mathbb C)} \\
	{{{\cal F}}} && {{{\cal E}}}
	\arrow[from=1-2, to=2-2]
	\arrow["{\pi_{{\cal F}}^{\mathbb C}}"', from=1-2, to=3-1]
	\arrow[""{name=0, anchor=center, inner sep=0}, "{\pi_f^{\mathbb C}}"{description, pos=0.3}, shift left=4, from=1-2, to=3-3]
	\arrow["{\pi_{{\cal F}}}"{description, pos=0.3}, from=2-2, to=4-1]
	\arrow[""{name=1, anchor=center, inner sep=0}, "{\pi_f}"', from=2-2, to=4-3]
	\arrow[from=3-1, to=4-1]
	\arrow[""{name=2, anchor=center, inner sep=0}, "{i_f^{\mathbb C}}"{description, pos=0.3}, from=3-3, to=1-2]
	\arrow["{F^{\mathbb C}}"', from=3-3, to=3-1]
	\arrow[from=3-3, to=4-3]
	\arrow[""{name=3, anchor=center, inner sep=0}, "{i_f}"', shift right=3, from=4-3, to=2-2]
	\arrow["{f^*}"', from=4-3, to=4-1]
	\arrow["\dashv"{anchor=center, rotate=-146}, draw=none, from=0, to=2]
	\arrow["\dashv"{anchor=center, rotate=34}, draw=none, from=1, to=3]
\end{tikzcd}\]

Because of \ref{pullbackdensemorphsites}, $\pi_{{\cal F}}^{\mathbb C}$ is a dense morphism of sites (for the Giraud topologies), and a comorphism inducing an equivalence of toposes (compatible with the structure morphisms) by \ref{pullbackdensemorphsites}. Moreover, $i_f^{\mathbb C}$ is a morphism of sites: indeed, it is right adjoint to $\pi_f^{\mathbb C}$ which is a comorphism of sites by \ref{pullbackcomorphalongfib}. Since $\pi_{{\cal F}}i_f=f^*$, Proposition \ref{functorialityofstructuralfunctorinverseimage} ensures that $\pi_{{\cal F}}^{\mathbb C}i_f^{\mathbb C} \simeq F^{\mathbb C}$. Since $\pi_{{\cal F}}^{\mathbb C}$ and $i_f^{\mathbb C}$ are morphisms of sites, their composition $F^{\mathbb C}$ also is.

Moreover, the geometric morphisms induced by $i_f$ and $i_f^{\mathbb C}$ as comorphisms can be identified with the geometric morphisms induced by $\pi_f$ and $\pi_f^{\mathbb C}$ as comorphisms, where the square of comorphisms of sites is commutative at the topos-level: hence, the following square induces a commutative square of geometric morphisms.

\[\begin{tikzcd}
	{\mathcal{G}(Lan_{i_f}{\mathbb C})} && {\mathcal{G}(\mathbb C)} \\
	{({{\cal F}}/ f^*)} && {{{\cal E}}}
	\arrow[""{name=0, anchor=center, inner sep=0}, "{\pi_f^{\mathbb C}}"{description, pos=0.3}, shift left=4, from=1-1, to=1-3]
	\arrow[from=1-1, to=2-1]
	\arrow[""{name=1, anchor=center, inner sep=0}, "{i_f^{\mathbb C}}"{description, pos=0.3}, from=1-3, to=1-1]
	\arrow[from=1-3, to=2-3]
	\arrow[""{name=2, anchor=center, inner sep=0}, "{\pi_f}"{description, pos=0.3}, shift left=2, from=2-1, to=2-3]
	\arrow[""{name=3, anchor=center, inner sep=0}, "{i_f}"{description, pos=0.3}, shift left=2, from=2-3, to=2-1]
	\arrow["\dashv"{anchor=center, rotate=-90}, draw=none, from=0, to=1]
	\arrow["\dashv"{anchor=center, rotate=-90}, draw=none, from=2, to=3]
\end{tikzcd}\]

The previous observation, combined with the fact that $\pi_{{\cal F}}$ and $\pi_{{\cal F}}^{\mathbb C}$ induce equivalences at the topos level compatible with the structure morphisms, implies that $F^{\mathbb C}$ is a morphism of sites inducing a commutative square at the topos level.
\end{proof}

\subsection{The first part of the equivalence}\label{subsec4}
Now that the existence of the expected projection for the pullback of toposes has been established, it becomes possible to verify that the equivalence of categories induced by the transposition of morphisms of fibrations (\ref{eq:weak_adjunction}) restricts, on one side, to the morphisms of fibrations that are also morphisms of sites.

Recall that, in order to prove that the relative presheaf topos on the inverse image of $\mathbb C$ along $f$ is the pullback of the relative presheaf topos on $\mathbb C$, the strategy is, for an arbitrary relative topos $g : {\cg} \to {{\cal F}}$, to obtain a correspondence between the geometric morphisms as dashed in the following diagram:

\[\begin{tikzcd}
	{{\cg}} \\
	& {\mathbf{Gir}( f^*\mathbb C)} & {\mathbf{Gir}(\mathbb C)} \\
	& {{\cal F}} & {{\cal E}}
	\arrow["{\overline{h}}"{description}, dashed, from=1-1, to=2-2]
	\arrow["h", bend left = 24, dashed, from=1-1, to=2-3]
	\arrow["g"', bend right = 24, from=1-1, to=3-2]
	\arrow["{f^{\mbc}}", from=2-2, to=2-3]
	\arrow["{C_{p'}}"', from=2-2, to=3-2]
	\arrow["{C_{p}}", from=2-3, to=3-3]
	\arrow["f"', from=3-2, to=3-3]
\end{tikzcd}\]

Recall that $f^{\mbc}$ has been defined in the previous subsection to be $\Sh(F^{\mbc})$, where $F^{\mbc}$ is the canonical arrow:

\[\begin{tikzcd}
	{\mathcal{G}(f^*\mathbb C)} & {\mathcal{G}(\mathbb C)} & {} \\
	{{{\cal F}}} & {{{\cal E}}}
	\arrow["{p'}"', from=1-1, to=2-1]
	\arrow["{F^{\mbc}}"', from=1-2, to=1-1]
	\arrow["p", from=1-2, to=2-2]
	\arrow["{f^*}"', from=2-2, to=2-1]
\end{tikzcd}\]

Let's study the correspondence between the relative geometric morphisms $h$ and $\overline{h}$: starting with a relative geometric morphism $\overline{h} : [g] \to [C_{p'}]$, one obtains a relative geometric morphism over ${\cal E}$ by postcomposition with the structural arrows $\overline{h}: [fg] \to [fC_{p'}]$; there remains to postcompose $\overline{h}$ with $f^{\mbc}$, yielding a relative geometric morphism $f^{\mbc}\overline{h} := h : [fg] \to [S_{C_{p}}]$. At the level of sites, this operation is provided by composing the morphism of sites $\overline{H}$ associated to the relative geometric morphism $\overline{h}$ with the structural morphism of sites $F^{\mbc}$, as pictured:

\[\begin{tikzcd}
	{({\cg}/g^*)} & {\mathcal{G}(f^*\mathbb C)} & {\mathcal{G}(\mathbb C)} \\
	& {{{\cal F}}} & {{{\cal E}}}
	\arrow["{\pi_g}"', from=1-1, to=2-2]
	\arrow["{\overline{H}}"', from=1-2, to=1-1]
	\arrow["{p'}"', from=1-2, to=2-2]
	\arrow["{F^{\mbc}}"', from=1-3, to=1-2]
	\arrow["p", from=1-3, to=2-3]
	\arrow["{f^*}"', from=2-3, to=2-2]
\end{tikzcd}\]

The goal is to verify that this construction indeed arises as the restriction of the transposition operation:

\[\begin{tikzcd}
	{\mathbf{Fib}_{{\cal E}'}([p'],[\pi_g])} & {\mathbf{Fib}_{{\cal E}}([p],[\pi_{fg}])} \\
	{\mathbf{FibSites}_{{\cal E}'}([p'],[\pi_g])} & {\mathbf{FibSites}_{{\cal E}}([p],[\pi_{fg}])}
	\arrow["\simeq"{description}, hook, from=1-1, to=1-2]
	\arrow["adjunction"{description}, shift left=3, draw=none, from=1-1, to=1-2]
	\arrow[hook, from=2-1, to=1-1]
	\arrow["restriction"{description}, shift left=3, draw=none, from=2-1, to=2-2]
	\arrow[shift left, hook, from=2-1, to=2-2]
	\arrow[hook, from=2-2, to=1-2]
\end{tikzcd}\]

Starting from an indexed weak geometric morphism $\overline{h} :S_g \to S_{C_{p'}}$ it can be turned into a morphism of indexed categories $\overline{H} : {f^* \mathbb C} \to S_g$ by indexed weak Diaconescu's theorem (Theorem 6.3.4. \cite{locfib}). If the indexed weak geometric morphism was an indexed geometric morphism, the morphism of fibrations $\overline{H}$ is moreover a morphism of sites (between the obvious relative sites) because of relative Diaconescu's theorem (Theorem 3.15. \cite{bartolicaramello}). Then, precomposing it with the morphism of sites $F^{\mbc} : \mathcal{G}(\mathbb C) \to \mathcal{G}(f^* \mathbb C)$ provides another morphism of sites $\overline{H}F^{\mathbb C} : \mathcal{G}(\mathbb C) \to ({\cg}/g^*)$, as depicted:

\[\begin{tikzcd}
	{({\cg}/g^*)} & {\mathcal{G}(f^*\mathbb C)} & {\mathcal{G}(\mathbb C)} \\
	& {{{\cal F}}} & {{{\cal E}}}
	\arrow["{\pi_g}"', from=1-1, to=2-2]
	\arrow["{\overline{H}}"', from=1-2, to=1-1]
	\arrow["{p'}"', from=1-2, to=2-2]
	\arrow["{F^{\mbc}}"', from=1-3, to=1-2]
	\arrow["p", from=1-3, to=2-3]
	\arrow["{f^*}"', from=2-3, to=2-2]
\end{tikzcd}\]

This morphism of sites induces a geometric morphism $\Sh(\overline{H}F^{\mathbb C})$ making the square commutative, because of the commutation of the following diagram:

\[\begin{tikzcd}
	{{\cg}} & {\mathbf{Gir}(f^*\mathbb C)} & {\mathbf{Gir}(\mathbb C)} \\
	& {{{\cal F}}} & {{{\cal E}}}
	\arrow["{\Sh(\overline{H})}", from=1-1, to=1-2]
	\arrow["{C_{\pi_g}}"', from=1-1, to=2-2]
	\arrow["{\Sh(F^{\mbc})}", from=1-2, to=1-3]
	\arrow["{C_{p'}}", from=1-2, to=2-2]
	\arrow["{C_p}", from=1-3, to=2-3]
	\arrow["f"', from=2-2, to=2-3]
\end{tikzcd}\]

Indeed, $\Sh(F^{\mbc})$ makes the square commutative because of \ref{F^Cmorphsites}, and $\overline{H}$ makes the triangle commutative because it is a morphism of sites and fibrations (Theorem 3.15. \cite{bartolicaramello}).

More precisely, we need the following Morita-equivalence result:

\begin{prop}\label{denseprojec}
Let $f : {{\cal F}} \to {{\cal E}}$ and $g : {\cg} \to {{\cal E}}$ be two relative toposes. The projection functor $q_{S_g}^{f^*}$ 

\[\begin{tikzcd}
	{({\cg}/g^*)} & {({\cg}/g^*f^*)} \\
	{{{\cal F}}} & {{{\cal E}}}
	\arrow["{\pi_g}"', from=1-1, to=2-1]
	\arrow["{q_{S_g}^{f^*}}"', from=1-2, to=1-1]
	\arrow["\lrcorner"{anchor=center, pos=0.125, rotate=-90}, draw=none, from=1-2, to=2-1]
	\arrow["{\pi_{fg}}", from=1-2, to=2-2]
	\arrow["{f^*}", from=2-2, to=2-1]
\end{tikzcd}\]
induces an equivalence between the two canonical relative sites, as depicted (modulo the Morita-equivalence between a relative topos and the topos on its canonical relative site of \ref{canonicalrelativesiteMoritaequi}):

\[\begin{tikzcd}
	{\widehat{{(\cg}/g^*)}_{J_g}} & {\widehat{({\cg}/g^*f^*)}_{J_{fg}}} \\
	{{{\cal F}}} & {{{\cal E}}}
	\arrow["{\Sh(q_{S_g}^{f^*})}", from=1-1, to=1-2]
	\arrow["\simeq"', draw=none, from=1-1, to=1-2]
	\arrow["g"', from=1-1, to=2-1]
	\arrow["fg", from=1-2, to=2-2]
	\arrow["f"', from=2-1, to=2-2]
\end{tikzcd}\]
\end{prop}

\begin{proof}
First, since the coverings are those that cover on the first component, it is immediate that $q_{S_g}^{f^*}$ preserves them, and since $f^*$ preserves finite limits, it implies that $q_{S_g}^{f^*}$ also does: $q_{S_g}^{f^*}$ is a morphism of sites.

It is immediate to verify that the following square of functors commutes:

\[\begin{tikzcd}
	{{\cg}} & {{\cg}} \\
	{({\cg}/g^*)} & {({\cg}/g^*f^*)}
	\arrow["{1_{\cg}}"', from=1-2, to=1-1]
	\arrow["{\pi_{\cg}}", from=2-1, to=1-1]
	\arrow["{\pi_{\cg}'}", from=2-2, to=1-2]
	\arrow["{q_{S_g}^{f^*}}", from=2-2, to=2-1]
\end{tikzcd}\]

Since $\pi_{\cg}$ and $\pi_{\cg}'$ are dense bimorphism of sites, it ensures that $q_{S_g}^{f^*}$ also induces an equivalence.

Moreover, recall that $ \pi_g \dashv \tau_g$ and $\pi_{fg} \dashv \tau_{fg}$, so that $\pi_{fg}$ and $\pi_g$ induce, as comorphisms, the same geometric morphisms as, respectively, $\tau_{fg}$ and $\tau_g$, as morphisms of sites. It is immediate to see that the following square is commuting:

\[\begin{tikzcd}
	{({\cg}/g^*)} & {({\cg}/g^*f^*)} \\
	{{{\cal F}}} & {{{\cal E}}}
	\arrow["\lrcorner"{anchor=center, pos=0.125}, draw=none, from=1-1, to=2-2]
	\arrow["{q_{S_g}^{f^*}}"', from=1-2, to=1-1]
	\arrow["{\tau_{g}}", from=2-1, to=1-1]
	\arrow["{\tau_{fg}}"', from=2-2, to=1-2]
	\arrow["{f^*}", from=2-2, to=2-1]
\end{tikzcd}\]

Hence, by functoriality of morphisms of sites, this yields the equivalence at the level of toposes, making the square commuting, as in the proposition.
\end{proof}

As previously, starting with an indexed weak geometric morphism $\overline{h} : S_g \to S_{C_{p'}}$ which also is an indexed geometric morphism, one can turn it into a morphism of sites $\overline{H} : \mathcal{G}(f^*\mathbb C) \to ({\cg} / g^*)$. One can transpose this morphism of sites by applying to it the direct image $f_*(\overline{H})$ and precomposing it with the unit $\varsigma_{\mathbb C}^{f}$ of the adjunction $f^* \dashv f_*$. In pictures:

\[\begin{tikzcd}
	{\mathcal{G}(f^*\mathbb C)} & {({\cg}/g^*)} \\
	{\mathcal{G}(f_*f^*\mathbb C)} & {({\cg}/g^*f^*)} \\
	{\mathcal{G}(\mathbb C)}
	\arrow["{\overline{H}}", from=1-1, to=1-2]
	\arrow[from=2-1, to=1-1]
	\arrow["{f_*(\overline{H})}", from=2-1, to=2-2]
	\arrow["{q_{S_g}^{f^*}}"', from=2-2, to=1-2]
	\arrow["{F^{\mbc}}", bend left = 55, from=3-1, to=1-1]
	\arrow["{\varsigma_{\mathbb C}^{f}}", from=3-1, to=2-1]
	\arrow["H"', from=3-1, to=2-2]
\end{tikzcd}\]

In this diagram, the transpose of $\overline{H}$ is denoted by $H$. Since $F^{\mbc}$ is a morphism of sites (\ref{F^Cmorphsites}) and the previous diagrams commutes, $q_{S_g}^{f^*}\circ H$ also is a morphism of sites (see Proposition \ref{densereflectmorphsites}). Finally, because of the previous proposition: $q_{S_g}^{f^*}\circ H$ being a morphism of sites and $q_{S_g}^{f^*}$ being a dense one implies $H$ also is a morphism of sites, so that the transposition operation, in this direction, sends morphisms of sites to morphisms of sites (or, equivalently, indexed weak geometric morphisms which are indexed geometric morphisms to indexed weak geometric morphisms which are indexed geometric morphisms).

This proof of the first direction is expressed by the restriction depicted in the following diagram:
\vspace{0.5cm}

\begin{adjustbox}{scale=0.5}
$\begin{tikzcd}[column sep=tiny]
	{{\mathbf{IndWeakGeom}_{{\cal F}}(S_g,S_{C_{p'}})}} &&& {\mathbf{IndWeakGeom}_{{\cal E}}(S_{fg},S_{C_{p}})} \\
	& {{\mathbf{IndGeom}_{{\cal F}}(S_g,S_{C_{p'}})}} & {{\mathbf{IndGeom}_{{\cal E}}(S_{fg},S_{C_{p}})}} \\
	& {\mathbf{FibSites}_{{\cal F}}((\mathcal{G}(f^*\mathbb C),Gir_{f^*\mathbb C}),(({\cg}/g^*),J_g))} & {\mathbf{FibSites}_{{\cal E}}((\mathcal{G}(\mathbb C),Gir_{\mathbb C}),(({\cg}/g^*f^*),J_{fg}))} \\
	{\mathbf{Fib}_{{\cal F}}({\cg}(f^*\mathbb C),({\cg}/g^*))} &&& {\mathbf{Fib}_{{\cal E}}({\cg}(\mathbb C),({\cg}/g^*f^*))}
	\arrow["\simeq"{description}, from=1-1, to=1-4]
	\arrow[hook, from=2-2, to=1-1]
	\arrow[hook', from=2-2, to=2-3]
	\arrow[hook', from=2-3, to=1-4]
	\arrow["\simeq"{description}, no head, from=3-2, to=2-2]
	\arrow[hook, from=3-2, to=3-3]
	\arrow[hook', from=3-2, to=4-1]
	\arrow["\simeq"{description}, no head, from=3-3, to=2-3]
	\arrow[hook, from=3-3, to=4-4]
	\arrow["\simeq"{description}, no head, from=4-1, to=1-1]
	\arrow["\simeq"{description}, from=4-1, to=4-4]
	\arrow["\simeq"{description}, no head, from=4-4, to=1-4]
\end{tikzcd}$
\end{adjustbox}
\vspace{0.5cm}

\noindent That is, the operation of transposition of morphisms of fibrations along $f^* \dashv f_*$ sends morphisms of sites on the left to morphisms of sites to the right. This is the following proposition:

\begin{prop}\label{firstdirection}
Let $f : \cf \to \ce$ and $g : \cg \to \cf$ be relative toposes, $\mbc$ a $\ce$-indexed category and $A : (\cg(f^*\mbc),Gir_{\mbc}) \to ((\cg/g^*),J_{g})$ a morphism of sites and fibrations. The transpose $A^t : (\cg(\mbc),Gir_{\mbc}) \to ((\cg/g^*f^*),J_{fg})$ of $A$ also is a morphism of sites.
\end{prop}\qed

\subsection{The other half of the equivalence}\label{subsec5}

In this subsection, in order to avoid the overload of upperscripts, the $\eta$-extension of a morphism of sites will be denoted without the upper $*$; for example the $\eta$-extension of a morphism of sites $A$ will be denoted as $\widetilde{A}$ - not $\widetilde{A}^*$.

The goal is now to establish the converse direction: namely, that transposition of morphisms of fibrations along the adjunction $f^* \dashv f_*$ sends morphisms of sites  
\[ H : (\mathcal{G}(\mathbb C), Gir_{\mbc}) \to (({\cg}/g^*f^*), J_{fg}) \]  
to morphisms of sites  
\[ H^t : (\mathcal{G}(f^*\mathbb C), Gir_{f^*\mbc}) \to (({\cg}/g^*), J_g). \]

Since $H$ is the restriction of its $\eta$-extension (that is, we can write $H$ as the composite $H \simeq \widetilde{H}\eta_{\mbc}$, see Proposition 3.4. (iv) \cite{bartolicaramello}), a first way to compute $H^t$ is by taking the transpose of $\widetilde{H}$ and the inverse image of $\eta_\mbc$: this yields $H^t \simeq \widetilde{H}^tf^*(\eta_\mbc)$. On the other hand, since $H^t$ is a morphism of fibrations, it is continuous for the Giraud topology on the domain: hence $H^t$ has an $\eta$-extension and this provides $H^t \simeq \widetilde{H^t}\eta_{f^*\mbc}$. These two ways of computing the transposition are pictured in the following commutative diagram:

\[\begin{tikzcd}
	& {S_g} \\
	{f^*(S_{\mathbb C})} && {S_{f^*\mathbb C}} \\
	& {f^*\mbc}
	\arrow["{\widetilde{H}^t}", from=2-1, to=1-2]
	\arrow["{\widetilde{H^t}}"', from=2-3, to=1-2]
	\arrow["{H^t}"{description}, from=3-2, to=1-2]
	\arrow["{f^*(\eta_{\mbc})}", from=3-2, to=2-1]
	\arrow["{\eta_{f^*\mbc}}"', from=3-2, to=2-3]
\end{tikzcd}\]

From $H$ being a morphism of relative sites, it follows that $\widetilde{H}$ is a cartesian morphism of fibrations (see Proposition 3.4. (iii) \cite{bartolicaramello}). Since the transpose of a cartesian morphism of fibrations is again cartesian, $\widetilde{H}^t$ is also a cartesian morphism of fibrations. The goal is now to equip $\cg(f^*(S_{\mbc}))$ with a suitable topology making $\widetilde{H}^t$ continuous (and hence a morphism of sites, since it also preserves finite limits), and ensuring that $f^*(\eta_\mbc)$ is a morphism of sites as well. This will imply that their composition  -  namely $H^t$  -  is a morphism of sites, which is the desired result.

The previous section already provides a commutative square of geometric morphisms:

\[\begin{tikzcd}
	{\mathbf{Gir}(f^*\mbc)} & {\mathbf{Gir}(\mbc)} \\
	\cf & \ce
	\arrow["{f^{\mbc}}", from=1-1, to=1-2]
	\arrow["{C_{p'}}"', from=1-1, to=2-1]
	\arrow["{C_p}", from=1-2, to=2-2]
	\arrow["f"', from=2-1, to=2-2]
\end{tikzcd}\]

In particular, $f^\mbc : [fC_{p'}] \to [C_p]$ is a morphism of relative toposes. By \ref{etaextensionbasechangeproperties}, it induces a morphism of cartesian indexed categories 
$$\widetilde{(f^\mbc)^*} : S_{C_p} \to S_{fC_{p'}}.$$ Since $S_{fC_{p'}} \simeq f_*(S_{C_{p'}})$ is the direct image of the indexed category $S_{C_{p'}}$ through the morphism of sites $f^*$, one can transpose $\widetilde{(f^\mbc)^*} : S_{C_p} \to S_{fC_{p'}}$ into a morphism of cartesian fibrations $\widetilde{(f^\mbc)^*}^t : f^*(S_{C_p}) \to S_{C_{p'}}$. This establishes a comparison functor between the inverse image of the canonical stack and the canonical stack of the inverse image. The following proposition establishes the good behavior of this functor with respect to $\eta$-extensions:

\begin{prop}\label{commutationcomparisonfunctor}
Let $f : \cf \to \ce$ and $g : \cg \to \cf$ be relative toposes, $\mbc$ a $\ce$-indexed category and $A : \mbc \to S_{fg}$ a morphism of indexed categories. The following is a commutative diagram of morphisms of indexed categories:

\[\begin{tikzcd}
	& {S_g} \\
	{f^*(S_{\mathbb C})} && {S_{f^*\mathbb C}} \\
	& {f^*\mbc}
	\arrow["{\widetilde{A}^t}", from=2-1, to=1-2]
	\arrow["{\widetilde{(f^\mbc)^*}^t}"{description}, from=2-1, to=2-3]
	\arrow["{\widetilde{A^t}}"', from=2-3, to=1-2]
	\arrow["{f^*(\eta_{\mbc})}", from=3-2, to=2-1]
	\arrow["{\eta_{f^*\mbc}}"', from=3-2, to=2-3]
\end{tikzcd}\]
\end{prop}

\begin{proof}
To check the commutativity of the upper and lower triangles, it suffices to verify that their transposes commute. This will be more convenient, since $\widetilde{(f^\mbc)^*}^t$ is defined as a transpose.

For the upper triangle, after transposition, the desired isomorphism is: $\widetilde{A}^* \simeq f_*(\widetilde{A^t})\circ \widetilde{(f^{\mbc})^*}^*$. In picture, this is the commutativity of the following triangle:

\[\begin{tikzcd}
	& {S_{fg}} \\
	{S_{C_p}} && {S_{fC_{p'}}}
	\arrow["{\widetilde{A}}", from=2-1, to=1-2]
	\arrow["{\widetilde{(f^\mbc)^*}}"{description}, from=2-1, to=2-3]
	\arrow["{f_*(\widetilde{A^t})}"', from=2-3, to=1-2]
\end{tikzcd}\]

First, observe that $A$ and $A^t F^{\mathbb C}$ induce the same weak geometric morphism. Indeed, it is immediate that

\[\begin{tikzcd}
	& {f^*\mathbb C} & {S_g} \\
	{\mathbb C} & {f_*f^*\mathbb C} & {S_{fg}}
	\arrow["{A^t}", from=1-2, to=1-3]
	\arrow["{F^{\mathbb C}}", from=2-1, to=1-2]
	\arrow["{\varsigma_{\mathbb C}^{f}}"', from=2-1, to=2-2]
	\arrow["A"', bend right = 39, from=2-1, to=2-3]
	\arrow["{q^{f^*}_{f^*\mathbb C}}", from=2-2, to=1-2]
	\arrow["{f_*(A^t)}"', from=2-2, to=2-3]
	\arrow["{q^{f^*}_{S_g}}"', from=2-3, to=1-3]
\end{tikzcd}\]
is a commutative diagram, so that $q^{f^*}_{S_g} A \simeq A^tF^{\mathbb C}$; but $q^{f^*}_{S_g}$ is a dense morphism of sites \ref{denseprojec}, so that $\Sh(A)\Sh(q^{f^*}_{S_g}) \simeq \Sh(F^{\mathbb C})\Sh(A^t)$. Hence, modulo the equivalence given by $\Sh(q^{f^*}_{S_g})$, the identity $\Sh(A) \simeq \widetilde{(f^{\mbc})^*}\Sh(A^t)$ holds, as $\Sh(F^{\mathbb C}) \simeq f^{\mbc}$. Based on this identity, a straightforward computation ensures the desired commutation: $\widetilde{A} \simeq f_*(\widetilde{A^t})\circ \widetilde{(f^{\mbc})^*}$. 

It remains to prove the commutativity of the lower triangle, that is:

\[\begin{tikzcd}
	{f^*(S_{\mathbb C})} && {S_{f^*\mathbb C}} \\
	& {f^*\mbc}
	\arrow["{\widetilde{(f^{\mbc})^*}^t}"{description}, from=1-1, to=1-3]
	\arrow["{f^*(\eta_{\mbc})}", from=2-2, to=1-1]
	\arrow["{\eta_{f^*\mbc}}"', from=2-2, to=1-3]
\end{tikzcd}\]

Again, since $\widetilde{(f^{\mbc})^*}^t$ is defined as a transpose, it is equivalent to check the commutativity of the transpose of the lower triangle, namely:
\[\begin{tikzcd}
	{S_{\mathbb C}} && {f_*(S_{f^*\mathbb C})} \\
	& \mbc
	\arrow["{\widetilde{(f^{\mbc})^*}}"{description}, from=1-1, to=1-3]
	\arrow["{\eta_{\mbc}}", from=2-2, to=1-1]
	\arrow["{(\eta_{f^*\mbc})^t}"', from=2-2, to=1-3]
\end{tikzcd}\]

In terms of fibrations, the previous triangle is equivalent to the following:

\[\begin{tikzcd}
	{(\mathbf{Gir}(\mbc)/C_p^*)} && {(\mathbf{Gir}(f^*\mbc)/C_{p'}^*f^*)} \\
	& {\cg(\mbc)}
	\arrow["{\widetilde{(f^\mbc)^*}}"{description}, from=1-1, to=1-3]
	\arrow["{\eta_{\mbc}}", from=2-2, to=1-1]
	\arrow["{(\eta_{f^*\mbc})^t}"', from=2-2, to=1-3]
\end{tikzcd}\]

The category $(\mathbf{Gir}(f^*\mbc)/C_{p'}^*f^*)$ is given as the following bipullback:

\[\begin{tikzcd}
	{(\mathbf{Gir}(f^*\mbc)/C_{p'}^*f^*)} & {(\mathbf{Gir}(f^*\mbc)/C_{p'}^*)} \\
	\ce & \cf
	\arrow[from=1-1, to=1-2]
	\arrow[from=1-1, to=2-1]
	\arrow["{\pi'}", from=1-2, to=2-2]
	\arrow[""{name=0, anchor=center, inner sep=0}, "{f^*}"', from=2-1, to=2-2]
	\arrow["\lrcorner"{anchor=center, pos=0.125}, draw=none, from=1-1, to=0]
\end{tikzcd}\]

By the universal property of the bipullback, the identity $\widetilde{(f^{\mbc})^*}\eta_{\mbc} \simeq (\eta_{f^*\mbc})^t$ holds if and only if the two sides agree after composition with the two projections. For the projection on ${\cal F}$ it is immediate, since they all are morphisms of fibrations (hence, commuting over the base category). For the second projection, it will be useful to look at the following diagram:

\[\begin{tikzcd}
	{\cg(\mbc)} & {\cg(f_*f^*\mbc)} & {\cg(f^*\mbc)} \\
	{(\mathbf{Gir}(\mbc)/C_p^*)} & {(\mathbf{Gir}(f^*\mbc)/C_{p'}^*f^*)} & {(\mathbf{Gir}(f^*\mbc)/C_{p'}^*)}
	\arrow["{\varsigma_{\mbc}^f}", from=1-1, to=1-2]
	\arrow["{F^{\mbc}}", bend left = 30, from=1-1, to=1-3]
	\arrow["{\eta_{\mbc}}", from=1-1, to=2-1]
	\arrow["{q^{f^*}_{f^*\mbc}}", from=1-2, to=1-3]
	\arrow["{f_*(\eta_{f^*\mbc})}", from=1-2, to=2-2]
	\arrow["{\eta_{f^*\mbc}}", from=1-3, to=2-3]
	\arrow["{\widetilde{(f^{\mbc})^*}_{f^*}}"', shift right, from=2-1, to=2-2]
	\arrow["{\widetilde{(f^{\mbc})^*}}"{description}, bend right = 30, from=2-1, to=2-3]
	\arrow["{q^{f^*}_{S_{f^*\mbc}}}"', shift right, from=2-2, to=2-3]
\end{tikzcd}\]

Notice that here it is necessary to distinguish between ${\widetilde{(f^{\mbc})^*}}$ and ${\widetilde{(f^{\mbc})^*}_{f^*}}$ (see Remark \ref{remA_B}).

Since $f^{\mbc}$ is induced by the morphism of sites $F^{\mbc}$, one has $\widetilde{(f^{\mbc})^*} \simeq \widetilde{F^{\mbc}}$. Then, by point (iv) of Proposition \ref{etaextensionbasechangeproperties}, the outer rectangle is commutative. The right square of the diagram is commutative by construction. What remains to be verified is that the two paths of the left square, when postcomposed with the projection $q^{f^*}_{S_{f^*\mbc}}$, yield (up to isomorphism) the same functor. 

It holds that $q^{f^*}_{S_{f^*\mbc}} \circ \widetilde{(f^{\mbc})^*}_{f^*} \circ \eta_{\mbc} \simeq \eta_{f^*\mbc} \circ F^{\mbc}$. Moreover, the commutativity of the right square ensures that $\eta_{f^*\mbc} \circ F^{\mbc} \simeq q^{f^*}_{S_{f^*\mbc}} \circ f_*(\eta_{f^*\mbc}) \circ \varsigma_{\mbc}^f$. Finally, since $\varsigma_{\mbc}^f$ is the unit of the adjunction $f^* \dashv f_*$, it follows that $q^{f^*}_{S_{f^*\mbc}} \circ f_*(\eta_{f^*\mbc}) \circ \varsigma_{\mbc}^f \simeq q^{f^*}_{S_{f^*\mbc}} \circ \eta_{f^*\mbc}^t$.

Putting all this together yields the identity $q^{f^*}_{S_{f^*\mbc}}\circ \eta_{f^*\mbc}^t \simeq q^{f^*}_{S_{f^*\mbc}} \circ \widetilde{(f^{\mbc})^*}_{f^*} \circ \eta_{\mbc}$, so that, as required, $\widetilde{(f^{\mbc})^*}_{f^*} \circ \eta_{\mbc} \simeq \eta_{f^*\mbc}^t$.
\end{proof}

Now that $\widetilde{(f^{\mbc})^*}^t$ has been shown to be compatible with both the $\eta$ functors and the $\eta$-extensions, the next step is to prove that, for a suitable topology, it induces an equivalence at the topos level. This is initiated by the following proposition:

\begin{prop}
Let $f : \cf \to \ce$ be a relative topos and $\mbc$ a $\ce$-indexed category. There is a topology $J_{C_{p'}}^{\widetilde{(f^{\mbc})^*}^t}$ on $\cg(f^*(S_{\mbc}))$, namely the induced topology as in Proposition 6.5. \cite{denseness}, making $\widetilde{(f^{\mbc})^*}^t$ a morphism of sites inducing a surjection at the topos-level.
\end{prop}

\begin{proof}
The existence of the topology follows from the fact that $\widetilde{(f^{\mbc})^*}^t$ is cartesian: by Proposition 6.5 of \cite{denseness}, one can define the topology $J_{C_{p'}}^{\widetilde{(f^{\mbc})^*}^t}$ by declaring as coverings those families whose image under $\widetilde{(f^{\mbc})^*}^t$ are coverings for $J_{C_{p'}}$. Then, by Theorem 6.3 (i) of \cite{denseness} and the very definition of this image topology, the morphism of sites
\[
\widetilde{(f^\mbc)^*}^t : (f^*(S_{\mbc}),J_{C_{p'}}^{\widetilde{(f^{\mbc})^*}^t}) \to (S_{f^*\mbc},J_{C_{p'}})
\]
induces a surjection at the topos level.
\end{proof}

\begin{remark}
This topology is the \emph{image} topology, that is, the subtopos $\widehat{\cg({f^*(S_{\mbc})})}_{J_{C_{p'}}^{\widetilde{(f^{\mbc})^*}^t}}$ of $\widehat{\cg({f^*(S_{\mbc})})}$ is the image of the subtopos $\widehat{\cg({S_{f^*\mbc})}}_{J_{C_{p'}}}$ through the geometric morphism $\Sh(\widetilde{(f^{\mbc})^*}) : \widehat{\cg({S_{f^*\mbc})}} \to \widehat{\cg({f^*(S_{\mbc})})}$ defined between the presheaf toposes. This is depicted in the following diagram:

\[\begin{tikzcd}
	{\widehat{\cg({f^*(S_{\mbc})})}_{J_{C_{p'}}^{\widetilde{(f^{\mbc})^*}^t}}} & {\widehat{\cg({S_{f^*\mbc})}}_{J_{C_{p'}}}} \\
	{\widehat{\cg({f^*(S_{\mbc})})}} & {\widehat{\cg({S_{f^*\mbc})}}}
	\arrow["i"', hook, from=1-1, to=2-1]
	\arrow["{\Sh(\widetilde{(f^{\mbc})^*})}"', two heads, from=1-2, to=1-1]
	\arrow["{i'}", hook', from=1-2, to=2-2]
	\arrow["{\Sh(\widetilde{(f^{\mbc})^*})}", from=2-2, to=2-1]
\end{tikzcd}\]

There is also a definition of an image topology for a functor $F : (\cc,J) \to \cc'$: it gives the topology $J'$ on $\cc'$ such that the subtopos it induces is the image of the subtopos $\widehat{\cc}_J$ through the comorphism $C_F$ (see Proposition 6.11 \cite{denseness}). One can show that, in our case, the image topology $J_{C_{p'}}^{\widetilde{(f^{\mbc})^*}^t}$ for the morphism of sites $\widetilde{(f^{\mbc})^*}$ is the same as the image topology for $f^*(\eta_{\mbc}) : (\cg(f^*\mbc),Gir_{f^*\mbc}) \to \cg(f^*(S_{\mbc}))$ seen as a comorphism of sites. 
\end{remark}

Now, the following central result can be deduced, asserting that the inverse image of the canonical stack and the canonical stack of the inverse image are Morita-equivalent for relative presheaf toposes:

\begin{prop}
Let $f : \cf \to \ce$ be a relative topos and $\mbc$ a $\ce$-indexed category. The functors $f^*(\eta_\mbc) : (\cg(f^*\mbc),Gir_{f^*\mbc}) \to (\cg(f^*(S_{\mbc})),J_{C_{p'}}^{\widetilde{(f^{\mbc})^*}^t})$ and $\widetilde{(f^{\mbc})^*}^t : (\cg(f^*(S_{\mbc})),J_{C_{p'}}^{\widetilde{(f^{\mbc})^*}^t}) \to (S_{f^*\mbc},J_{C_{p'}})$ are morphisms of sites inducing equivalences at the topos-level.
\end{prop}

\begin{proof}
The image topology $J_{C_{p'}}^{\widetilde{(f^{\mbc})^*}^t}$ contains Giraud's one: indeed, $\widetilde{(f^{\mbc})^*}^t$ is a morphism of fibrations, and since $J_{C_{p'}}$ includes the Giraud topology, the very definition of $J_{C_{p'}}^{\widetilde{(f^{\mbc})^*}^t}$ forces $J_{C_{p'}}^{\widetilde{(f^{\mbc})^*}^t}$ to include it as well. Therefore, since $f^*(\eta_\mbc)$ is a morphism of fibrations and the topology on its domain is the Giraud one, it follows that it is continuous (see Proposition \ref{trivialsitescont}).

Now, since $\eta_{f^*\mbc} \simeq \widetilde{(f^{\mbc})^*}^t \circ f^*(\eta_\mbc)$, the isomorphism 
$$\Sh(\eta_{f^*\mbc})^* \simeq \Sh(\widetilde{(f^{\mbc})^*}^t)^* \circ \Sh(f^*(\eta_\mbc))^*$$
yields. Since $\Sh(\eta_{f^*\mbc})^*$ is an equivalence (see Proposition \ref{etadense}),  $\Sh(\widetilde{(f^{\mbc})^*}^t)^*$ must be essentially surjective and full. As the previous proposition ensures that $\Sh(\widetilde{(f^{\mbc})^*}^t)^*$ is a surjection, it is also faithful  -  hence an equivalence.

Finally, combining the identity $\Sh(\eta_{f^*\mbc})^* \simeq \Sh(\widetilde{(f^{\mbc})^*}^t)^* \circ \Sh(f^*(\eta_\mbc))^*$ with the fact that both $\Sh(\eta_{f^*\mbc})^*$ and $\Sh(\widetilde{(f^{\mbc})^*}^t)^*$ are equivalences, one can deduce that $\Sh(f^*(\eta_\mbc))^*$ is also an equivalence. This shows that $f^*(\eta_\mbc)$ also is a morphism of sites that induces an equivalence of toposes.
\end{proof}

Finally, this fact allows to deduce that the transpose of a morphism of sites and fibrations also is a morphism of sites and fibrations:

\begin{prop}\label{seconddirection}
Let $f : \cf \to \ce$ and $g : \cg \to \cf$ be relative toposes, $\mbc$ a $\ce$-indexed category, and $A : (\cg(\mbc),Gir_{\mbc}) \to (\cg(S_{fg}),J_{fg})$ a morphism of sites and fibrations. The transpose $A^t : (\cg(f^*\mbc),Gir_{f^*\mbc}) \to (\cg(S_g),J_g)$ of $A$ is a morphism of sites.
\end{prop}

\begin{proof}
The proof lies in the following commutative diagram (see Proposition \ref{commutationcomparisonfunctor}):

\[\begin{tikzcd}
	& {(S_g,J_g)} \\
	{(f^*(S_{\mathbb C}),J_{C_{p'}}^{\widetilde{(f^\mbc)^*}^t})} && {(S_{f^*\mathbb C},J_{C_{p'}})} \\
	& {(f^*\mbc,Gir_{f^*\mbc})}
	\arrow["{\widetilde{H}^t}", from=2-1, to=1-2]
	\arrow["{\widetilde{(f^\mbc)^*}^t}"{description}, from=2-1, to=2-3]
	\arrow["{\widetilde{H^t}}"', from=2-3, to=1-2]
	\arrow["{f^*(\eta_{\mbc})}", from=3-2, to=2-1]
	\arrow["{\eta_{f^*\mbc}}"', from=3-2, to=2-3]
\end{tikzcd}\]

Since $\widetilde{H}^t \simeq \widetilde{H^t} \circ \widetilde{(f^{\mbc})^*}^t$, the functor $\widetilde{H}^t$ is cover-preserving, as it is the composition of cover-preserving functors. It is also cartesian, and thus qualifies as a morphism of sites. Given that $H^t \simeq \widetilde{H}^t \circ f^*(\eta_\mbc)$ and that $f^*(\eta_\mbc)$ is a (dense) morphism of sites (by the previous proposition), it follows that $H^t$ is itself a morphism of sites.
\end{proof}

This just proves the converse of the previous subsection, that is: the other direction of the transposition along $f^* \dashv f_*$ also restricts to those morphisms of fibrations which are morphisms of sites:
\vspace{0.5cm}

\begin{adjustbox}{scale=0.5}
$\begin{tikzcd}[column sep=tiny]
	{{\mathbf{IndWeakGeom}_{{\cal F}}(S_g,S_{C_{p'}})}} &&& {\mathbf{IndWeakGeom}_{{\cal E}}(S_{fg},S_{C_{p}})} \\
	& {{\mathbf{IndGeom}_{{\cal F}}(S_g,S_{C_{p'}})}} & {{\mathbf{IndGeom}_{{\cal E}}(S_{fg},S_{C_{p}})}} \\
	& {\mathbf{FibSites}_{{\cal F}}((\mathcal{G}(f^*\mathbb C),Gir_{f^*\mathbb C}),(({\cg}/g^*),J_g))} & {\mathbf{FibSites}_{{\cal E}}((\mathcal{G}(\mathbb C),Gir_{\mathbb C}),(({\cg}/g^*f^*),J_{fg}))} \\
	{\mathbf{Fib}_{{\cal F}}({\cg}(f^*\mathbb C),({\cg}/g^*))} &&& {\mathbf{Fib}_{{\cal E}}({\cg}(\mathbb C),({\cg}/g^*f^*))}
	\arrow["\simeq"{description}, from=1-4, to=1-1]
	\arrow[hook', from=2-2, to=1-1]
	\arrow[hook, from=2-3, to=1-4]
	\arrow[hook', from=2-3, to=2-2]
	\arrow["\simeq"{description}, no head, from=3-2, to=2-2]
	\arrow[hook, from=3-2, to=4-1]
	\arrow["\simeq"{description}, no head, from=3-3, to=2-3]
	\arrow[hook', from=3-3, to=3-2]
	\arrow[hook', from=3-3, to=4-4]
	\arrow["\simeq"{description}, no head, from=4-1, to=1-1]
	\arrow["\simeq"{description}, no head, from=4-4, to=1-4]
	\arrow["\simeq"{description}, from=4-4, to=4-1]
\end{tikzcd}$
\end{adjustbox}
\vspace{0.3cm}

\subsection{Main theorem}\label{subsec6}

In this subsection, we synthesize the results of the previous subsections, deriving a universal property for the inverse image of the canonical site, together with the main theorem on the pullback of relative presheaf toposes. 

Equipped with the image topology, the inverse image of the canonical stack, namely the site $(\cg(f^*(S_{\mbc})),J_{C_{p'}}^{\widetilde{(f^{\mbc})^*}^t})$, plays the same role as the canonical relative site of the inverse image $(S_{f^*\mathbb C},J_{C_{p'}})$. Indeed, this relative site has the following universal property:

\begin{prop}
let $f : {{\cal F}} \to {{\cal E}}$ be a geometric morphism, $\mathbb C$ a ${{\cal E}}$-indexed category, and $g : {\cg} \to {{\cal F}}$ another relative topos. There is an equivalence of categories:  $$\mathbf{SitesFib}_{\cf}((\mathcal{G}(f^*(S_{\mbc})),J_{C_{p'}}^{\widetilde{(f^\mbc)^*}^t}),(({\cg}/g^*),J_g)) \simeq \mathbf{Top}/{{\cal F}}([g],[C_{p'}])$$
\end{prop}

\begin{proof}

Morphisms of sites and fibrations $$\overline{H}' : (\cg(f^*(S_{\mbc})),J_{C_{p'}}^{\widetilde{(f^{\mbc})^*}^t}) \to (({\cg}/g^*),J_g)$$ correspond to morphisms of sites and fibrations $$\overline{H} : \mathcal{G}(f^*(\mathbb C)) \to (({\cg}/g^*),J_{g})$$  (or, equivalently, to morphisms of relative toposes  $[fg] \to [C_p]$) in the way described below.  

A morphism of sites and fibrations $$\overline{H}' : (\cg(f^*(S_{\mbc})),J_{C_{p'}}^{\widetilde{(f^{\mbc})^*}^t}) \to (({\cg}/g^*),J_g)$$ is precomposed with $f^*(\eta_{\mathbb C})$ which is a (dense) morphism of sites and fibrations (Proposition \ref{commutationcomparisonfunctor}): this provides a morphism of sites and fibrations $\overline{H}'f^*(\eta_{\mbc}) : (\cg(f^*\mbc),Gir_{f^*\mbc}) \to ((\cg/g^*),J_g)$, and, by relative Diaconescu's theorem, this is equivalently a morphism of relative toposes $[g] \to [C_{p'}]$.

In the other direction, a morphism of sites and fibrations 
$$\overline{H} : (\mathcal{G}(f^*(\mathbb C)),Gir_{f^*\mbc}) \to (({\cg}/g^*),J_{g})$$ 
corresponds to a morphism of sites and fibrations
$$\overline{H}^t : (\mathcal{G}(\mathbb C),Gir_{\mbc}) \to (({\cg}/g^*f^*),J_{fg})$$
(see Proposition \ref{seconddirection}). By Proposition \ref{commutationcomparisonfunctor} applied to the morphism of sites and fibrations $\overline{H}^t$, the following diagram commutes:

\[\begin{tikzcd}
	& {(S_g,J_g)} \\
	{(f^*(S_{\mathbb C}),J_{C_{p'}}^{\widetilde{(f^\mbc)^*}^t})} && {(S_{f^*\mathbb C},J_{C_{p'}})} \\
	& {(f^*\mbc,Gir_{f^*\mbc})}
	\arrow["{\widetilde{\overline{H}^t}^t}", from=2-1, to=1-2]
	\arrow["{\widetilde{(f^\mbc)^*}^t}"{description}, from=2-1, to=2-3]
	\arrow["{\widetilde{\overline{H}}}"', from=2-3, to=1-2]
	\arrow["{f^*(\eta_{\mbc})}", from=3-2, to=2-1]
	\arrow["{\eta_{f^*\mbc}}"', from=3-2, to=2-3]
\end{tikzcd}\]

Since $\overline{H}$ is a morphism of sites, $\widetilde{\overline{H}}$ also is. Hence, $\widetilde{\overline{H}}\circ \widetilde{(f^\mbc)^*}^t \simeq \widetilde{\overline{H}^t}^t$ is a morphism of sites. 

One can easily show that these two correspondences actually constitute an equivalence, in the light of ${\widetilde{(f^\mbc)^*}^t}$ being a dense morphism of sites.
\end{proof}

Now, combining \ref{firstdirection} and \ref{seconddirection}:

\begin{prop}
Let $f : \cf \to \ce$ and $g : \cg \to \cf$ be two relative toposes, and $\mbc$ a $\ce$-indexed category. Transposition of morphisms of fibrations along $f^* \dashv f_*$ restricts to an equivalence of categories:

$$\mathbf{FibSites}_{\cf}((\cg(f^*\mbc),Gir_{f^*\mbc}),((\cg/g^*),J_g))$$
$$\simeq$$
$$\mathbf{FibSites}_{\ce}((\cg(\mbc),Gir_{\mbc}),((\cg/g^*f^*),J_{fg}))$$
\end{prop}
\qed

The full picture of the equivalences is finally depicted in the diagram:
\vspace{0.5cm}

\begin{adjustbox}{scale=0.5}
$\begin{tikzcd}[column sep=tiny]
	{{\mathbf{IndWeakGeom}_{{\cal F}}(S_g,S_{C_{p'}})}} &&& {\mathbf{IndWeakGeom}_{{\cal E}}(S_{fg},S_{C_{p}})} \\
	& {{\mathbf{IndGeom}_{{\cal F}}(S_g,S_{C_{p'}})}} & {{\mathbf{IndGeom}_{{\cal E}}(S_{fg},S_{C_{p}})}} \\
	& {\mathbf{FibSites}_{{\cal F}}((\mathcal{G}(f^*\mathbb C),Gir_{f^*\mathbb C}),(({\cg}/g^*),J_g))} & {\mathbf{FibSites}_{{\cal E}}((\mathcal{G}(\mathbb C),Gir_{\mathbb C}),(({\cg}/g^*f^*),J_{fg}))} \\
	{\mathbf{Fib}_{{\cal F}}({\cg}(f^*\mathbb C),({\cg}/g^*))} &&& {\mathbf{Fib}_{{\cal E}}({\cg}(\mathbb C),({\cg}/g^*f^*))}
	\arrow["\simeq"{description}, no head, from=1-4, to=1-1]
	\arrow[hook', from=2-2, to=1-1]
	\arrow[hook, from=2-3, to=1-4]
	\arrow["\simeq"', no head, from=2-3, to=2-2]
	\arrow[no head, from=3-2, to=2-2]
	\arrow[hook, from=3-2, to=4-1]
	\arrow[no head, from=3-3, to=2-3]
	\arrow["\simeq"', no head, from=3-3, to=3-2]
	\arrow[hook', from=3-3, to=4-4]
	\arrow["\simeq"{description}, no head, from=4-1, to=1-1]
	\arrow["\simeq"{description}, no head, from=4-1, to=4-4]
	\arrow["\simeq"{description}, no head, from=4-4, to=1-4]
\end{tikzcd}$
\end{adjustbox}
\vspace{0.4cm}

\noindent This allows to deduce the main theorem of the article:

\begin{prop}
Let $f : \cf \to \ce$ be a relative topos and $\mbc$ a $\ce$-indexed category. The following commutative square of geometric morphisms is a bipullback of toposes:

\[\begin{tikzcd}
	{\mathbf{Gir}(f^*\mbc)} & {\mathbf{Gir}(\mbc)} \\
	\cf & \ce
	\arrow["{f^*{\mbc}}", from=1-1, to=1-2]
	\arrow["{C_{p'}}"', from=1-1, to=2-1]
	\arrow["\lrcorner"{anchor=center, pos=0.125}, draw=none, from=1-1, to=2-2]
	\arrow["{C_p}", from=1-2, to=2-2]
	\arrow["f"', from=2-1, to=2-2]
\end{tikzcd}\]
\end{prop}

\begin{proof}
Indeed, the previous proposition combined with the relative Diaconescu's theorem yields an equivalence:

$${\mathbf{IndGeom}_{{\cal E}}(S_{fg},S_{C_{p}})} \simeq {{\mathbf{IndGeom}_{{\cal F}}(S_g,S_{C_{p'}})}} $$

Now recall the equivalence 

$$\mathbf{IndGeom}_{{\cal E}}(S_f,S_{f'}) \simeq \mathbf{Top}/{{\cal E}}([f],[f'])$$

for any relative toposes $f: {{\cal F}} \to {{\cal E}}$ and $f' : {{\cal F}}' \to {{\cal E}}$ (see subsection 6.1 \cite{locfib}). So, for an arbitrary relative topos $g :{\cg} \to {{\cal F}}$, the equivalence $\mathbf{IndGeom}_{{\cal F}}(S_g,S_{C_{p'}}) \simeq \mathbf{IndGeom}_{{\cal E}}(S_{fg},S_{C_{p}})$ induces the equivalence

$$\mathbf{Top}/{{\cal E}}([fg],[C_{p}]) \simeq {{\mathbf{Top}/{{\cal F}}([g],[C_{p'}])}}$$ 
which is the desired universal property.
\end{proof}

\vspace{1cm}

\bibliographystyle{alpha}
\bibliography{LBib}

\textsc{Léo Bartoli} 

\vspace{0.2cm}
{\small \textsc{Department of Mathematics, ETH Zurich, Rämistrasse 101
8092 Zurich, Switzerland.}\\
\emph{E-mail address:} \texttt{lbartoli@ethz.ch}

\vspace{0.2cm}

{\small \textsc{Istituto Grothendieck ETS, Corso Statuto 24, 12084 Mondovì, Italy.}\\
	\emph{E-mail address:} \texttt{leo.bartoli@ctta.igrothendieck.org}}

\vspace{0.6cm}

\textsc{Olivia Caramello} 

\vspace{0.2cm}
{\small \textsc{Dipartimento di Scienza e Alta Tecnologia, Universit\`a degli Studi dell'Insubria, via Valleggio 11, 22100 Como, Italy.}\\
	\emph{E-mail address:} \texttt{olivia.caramello@uninsubria.it}}

\vspace{0.2cm}

{\small \textsc{Istituto Grothendieck ETS, Corso Statuto 24, 12084 Mondovì, Italy.}\\
	\emph{E-mail address:} \texttt{olivia.caramello@igrothendieck.org}}

\end{document}